\documentclass[11pt]{article}
\usepackage{bbm}
\usepackage{authblk}
\usepackage[a4paper, margin=1in]{geometry}
\usepackage{amsmath, amssymb, amsthm, stmaryrd, bbold,mathtools, mathrsfs}
\usepackage{extarrows}
\usepackage{graphicx}
\usepackage{tikz, pgfplots,tikz-cd}
\usetikzlibrary{calc, decorations.markings, decorations.pathmorphing,shapes}
\usepackage{diagbox}
\usepackage{float}
\usepackage{multirow}
\usepackage{tabularray}
\usepackage{enumitem}
\usepackage{setspace}
\usepackage{ragged2e}
\usepackage[colorlinks, citecolor=YaleBlue, linkcolor=YaleBlue, urlcolor=YaleBlue, breaklinks=true]{hyperref}
\usepackage{enumitem,xcolor}
\usepackage[acronym]{glossaries}
\usepackage{mdframed}
\usepackage{lineno}
\allowdisplaybreaks[4]
\pgfplotsset{compat=1.18}
\usepackage{makecell}
\usepackage{setspace}
\usepackage{bibentry}
\usepackage{hyperref}
\nobibliography*

\renewcommand{\footnoterule}{%
  \kern-3pt
  \hrule width 2in height 0.4pt
  \kern 0pt
}
\usepackage{epigraph}
\definecolor{YaleBlue} {HTML}{00356b}

\definecolor{YaleMidBlue} {HTML}{286dc0}
\definecolor{YaleLightBlue} {HTML}{e5f0fd}
\definecolor{YaleMidGrey} {HTML}{b2b2b2}
\definecolor{YaleLightGrey} {HTML}{fcfcfc}

\definecolor{ao}{rgb}{117, 15, 109}

\definecolor{YaleBlack} {HTML}{222222}
\definecolor{YaleDarkGrey} {HTML}{4a4a4a}
\definecolor{YaleWhite} {HTML}{f9f9f9}
\definecolor{YaleLightGreen}{HTML}{AFB896}
\definecolor{YaleGreen} {HTML}{5f712d}
\definecolor{YaleOrange} {HTML}{bd5319}

\newcommand{\spd}[1]{(#1+1)d}

\theoremstyle{definition}
\newtheorem{theorem}{Theorem}[section]
\newtheorem{corollary}[theorem]{Corollary}
\newtheorem{prop}[theorem]{Proposition}

\theoremstyle{remark}

\newtheorem{example}{Example}

\newcommand{\Z}{\mathbb{Z}}

\newcommand{\C}{\mathbb{C}}

\newcommand{\cC}{ {\cal C} } 
\newcommand{\cD}{ {\cal D} }

\newcommand{\cH}{ {\cal H} }

\newcommand{\cM}{ {\cal M} }

\newcommand{\cX}{ {\cal X} } 
 
\newcommand{\cZ}{ {\cal Z} }

\newcommand{\sD}{ {\mathsf D} }

\newcommand{\sH}{ {\mathsf H} }

\newcommand{\Hilb}{\mathsf{Hilb}}

\newcommand{\one}{\mathbf{1}}

\newcommand{\DHR}{\mathrm{DHR}}

\newcommand{\Aut}{\mathrm{Aut}}
\newcommand{\Fun}{\mathsf{Fun}}

\newcommand{\Rep}{\mathsf{Rep}}

\newcommand{\id}{\mathrm{id}}

\newcommand{\TY}{\mathrm{TY}}

\newcommand{\ot}[1][]{\underset{#1}{\otimes}}

\newcommand{\fgt}{\mathrm{fgt}}

\tikzstyle string=[thin,postaction={decorate},decoration={markings,
    mark=at position 0.5*\pgfdecoratedpathlength+.5*4pt with {\arrow[scale=1]{stealth}}}]
    \tikzstyle revstring=[thin,postaction={decorate},decoration={markings,
    mark=at position 0.5*\pgfdecoratedpathlength+.5*4pt with {\arrowreversed[scale=1]{stealth}}}]
\newcommand{\diagram}[2]{
\begin{tikzpicture}[baseline=(current bounding box),scale=#1]
#2
\end{tikzpicture}
}

\tikzstyle trivial=[dashed, thick]

\tikzset{
  symbol/.style={
    draw=none,
    every to/.append style={
      edge node={node [sloped, allow upside down, auto=false]{$#1$}}}
  }
}

\newtheorem{thm}{Theorem}[section]

\newcommand{\googlebooks}[1]{(preview at \href{https://books.google.com/books?id=#1}{google books})}

\newcommand{\numdam}[1]{}

\theoremstyle{remark}

\theoremstyle{definition}

\newtheorem{defn}[thm]{Definition}
\newtheorem{quest}[thm]{Question}

\DeclareMathOperator{\End}{End}

\usepackage{graphicx,tikz}
\usetikzlibrary{arrows,backgrounds,scopes}
\usetikzlibrary{calc,cd}
\newcommand{\tikzmath}[2][0.42]
{\vcenter{\hbox{\begin{tikzpicture}[scale=#1] #2\end{tikzpicture}}}
}
\tikzset{coupon/.style={rectangle,rounded corners=1.5pt,draw,fill=white,inner sep=1.5,minimum size=12pt}}
\newcommand{\mydotw}[1]{\begin{scope}[shift={#1}] \fill[shift only,white] (0,0) circle (1.5pt); \draw[shift only,thick]  (0,0) circle (1.5pt);   \end{scope}}

\title{On the structure of categorical duality operators }

\author{Corey Jones\thanks{Department of Mathematics, North Carolina State University, Raleigh, NC 27695, USA}\,, Xinping Yang\thanks{Perimeter Institute for Theoretical Physics, Waterloo, ON N2L 2Y5, Canada}}

\date{}

\begin{document}

\maketitle

\begin{abstract}
We systematically study categorical duality operators on spin (and anyon) chains with respect to an internal fusion category symmetry $\mathcal{C}$. We parameterize duality operators on the quasi-local algebra in terms of data dependent on the associated quantum cellular automata (QCA) on the symmetric subalgebra $B$. In particular, a QCA $\alpha$ on $B$ defines an invertible $\mathcal{C}$-$\mathcal{C}$ bimodule category $\mathcal{M}_{\alpha}$, and the duality operators extending $\alpha$ form a simplex, with extreme points in bijective correspondence with the simple object of $\mathcal{M}_{\alpha}$. Then we consider the structure of external symmetries generated by a family of duality operators, and show that if the UV models are all defined on tensor product Hilbert spaces, these categories necessarily flow to weakly integral fusion categories in the IR.
\end{abstract}

\tableofcontents

\section{Introduction}

Of recent interest are generalized symmetries of spin chains \cite{seifnashri2024clusterstatenoninvertiblesymmetry,meng2025noninvertiblesptsonsiterealization,Inamura_2022, MR4861493,PhysRevB.111.054432,bhardwaj2025latticemodelsphasestransitions}. Fusion category symmetries, describing finite internal symmetries of quantum field theories in \spd{1}, are now fairly well-understood from a structural viewpoint using the framework of SymTFT \cite{Kong_2017,freed2024topologicalsymmetryquantumfield}, which has been recently developed on the lattice \cite{PhysRevB.111.054432,bhardwaj2025latticemodelsphasestransitions,evans2025operatoralgebraicapproachfusion}. In the operator algebraic framework for the infinite volume limit, SymTFT decompositions can be completely characterized by a the symmetric subalgebra $B$ (or in the SymTFT context, the physical boundary subalgebra) of the algebra $A$ of all quasi-local operators, and their DHR bimodules \cite{jones2024dhrbimodulesquasilocalalgebras,hataishi2025structuredhrbimodulesabstract}. This has led to formal interpretations of the categorical Landau paradigm \cite{bhardwaj2023categoricallandauparadigmgapped}, as well as versions of several anomaly-enforced gaplessness results \cite{Bhardwaj_2025,evans2025operatoralgebraicapproachfusion}.

In the context of fusion category symmetry, there is a very interesting phenomenon generalizing the famous Kramers-Wannier duality called \textit{categorical duality} \cite{cuiper2025gaugingdualityonedimensionalquantum,lu2026generalizedkramerswannierselfdualityhopfising,inamura2026remarksnoninvertiblesymmetriestensor,Lootens_2023,Lootens_2024}. A categorical duality operator is a (non-local), locality preserving operator\footnote{in the infinite volume limit, non-local operators can be characterized as arbitrary completely positive maps on the quasi-local algebra.} which is not necessarily unitary but restricts to a unitary on the symmetric sector. A duality operator can thus exchange inequivalent symmetric gapped phases. In the case of Kramers-Wannier, the underlying symmetry is the on-site $\mathbbm{Z}_{2}$ spin flip symmetry, and the Kramers-Wannier operator exchanges the trivial SPT with the symmetry breaking phase. This allows us in principle to systematically access non-trivial symmetric gapped phases from simpler ones. Thus an understanding of categorical dualities, and in particular classifying them up to symmetric finite-depth circuits, is a very interesting problem.

Any duality operator defines a quantum cellular automata (QCA) on the algebra $B$. Using the theory of DHR bimodules (outlined below), we prove the following theorem, which reduces the study of duality channels to QCA on the symmetric subalgebra $B$, which in turn have been extensively studied \cite{jones2024dhrbimodulesquasilocalalgebras, jones2025quantumcellularautomatacategorical, ma2024quantumcellularautomatasymmetric}.

\begin{theorem}\label{Theorem1} Suppose we have a fusion category $\mathcal{C}$ acting on the lattice, with symmetric subalgebra $B\subseteq A$. 

\begin{enumerate}
 \item 
Associated to $\alpha\in \text{QCA}(B)$ is an invertible $\mathcal{C}$-$\mathcal{C}$ bimodule category $\mathcal{M}_{\alpha}$.
\item
The set of (unital) duality operators which restrict to $\alpha$ on $B$ form a simplex, and the extreme points are in bijective correspondence with the simple objects of $\mathcal{M}_{\alpha}$.
\end{enumerate}
\end{theorem}

\noindent In particular, any classification of QCA on $B$ up to symmetric circuits provides a classification of duality channels up to symmetric circuits.

Another major motivation for studying Kramers-Wannier type dualities is their role in \textit{extending} the paradigm for non-invertible symmetry from the SymTFT picture described above to something more general. This larger class of symmetries, motivated explicitly by the Kramers-Wannier example, can be realized by duality operators but the fusion rules of these operators might ``mix with translation". This gives rise to infinite fusion rules on the lattice, which nevertheless flow to finite fusion rules in the IR \cite{Seiberg_2024}. The class of symmetries realized this way has been termed ``emanant" \cite{Cheng_2023,Seiberg:2023cdc}. 

An interesting observation is the class of fusion categories that can arise emanantly in the IR is strictly larger than the class that can be realized directly in the UV. Indeed the Kramers-Wannier operator flows under RG to fusion category symmetry with Ising fusion rules, which has an object with quantum dimension $\sqrt{2}$. Non-integral fusion categories cannot be realized as internal symmetries on the tensor product Hilbert space lattice in the UV \cite{evans2025operatoralgebraicapproachfusion}. In fact, it has been conjectured that any of these generalized fusion categorical symmetries emerging in the IR must be weakly integral, namely that the squares of the dimensions of all simple objects are integers \cite{lu2026generalizedkramerswannierselfdualityhopfising,inamura2026remarksnoninvertiblesymmetriestensor}. Part of the motivation of this paper is to demonstrate this conjecture is true, under the assumption that the generalized symmetry operators on the lattice are duality operators with respect to a fixed internal symmetry, extending the setup of the Kramers-Wannier operator.

When viewing duality operators as symmetries in the UV, we will use the term \textit{external symmetries} to contrast with the initial internal symmetries constraining RG. Operationally, a duality operator is external if it acts by a non-trivial QCA on the boundary algebra. External symmetries are designed to include not only spatial symmetries but also our duality operators. However, it is not entirely clear to what extent external symmetries are described by an explicit tensor category, rather than just a fusion ring (or more precisely a hypergroup). 

One point that distinguished external symmetries from internal ones is that the UV fusion rules on the lattice do not, in general, match the IR fusion fusion rules, even when the category in question is anomaly-free (equipped with a fiber functor). Moreover, external symmetries described by duality operators cannot necessarily be consistently defined on any fixed finite region in a quantum spin chain. For example, the non-invertible symmetry in $\Rep^\dagger(D_8)$ SPTs can be realized as the $\Z_2 \times \Z_2$ duality operator on the cluster spin chain \cite{seifnashri2024clusterstatenoninvertiblesymmetry}. While the lattice symmetry operators satisfy the $\Rep^\dagger(D_8)$ fusion rules, the duality operators do not satisfy the zipping condition \cite{lu2025strangecorrelatorstringorder}, and thus do not naturally form a tensor category in the UV. When the IR fusion category $\cX$ is anomalous, the UV external symmetry structure necessarily digresses from $\cX$ even more, meaning that the fusion rules \textit{must} mix with translation.

To clarify this situation, we argue that any (internal or external) symmetries realized either in the UV or IR by duality operators must necessarily be a quotient of a universal tensor category, which we construct. As a corollary, we argue that any emanant fusion category symmetry emergin in the IR from a symmetry constrained RG flow of external symmetries on a tensor product Hilbert space lattice must be weakly integral.

\begin{theorem}\label{Theorem:UniversalCat}
Let $B\subseteq A$ be a physical boundary subalgebra with fusion category $\mathcal{C}$, and suppose $\Phi_{1}, \dots \Phi_{n}$ are duality operators generating an external symmetry. Then there is a canonical $\mathbb{F}_{n}$-graded extension of the category $\mathcal{C}$, denoted as $\mathcal{C}\#\mathbb{F}_{n}$.
\begin{enumerate}
\item 
The normalized fusion rules associated to $\{\Phi_{i}\}$ are a quotient of the fusion rules of the tensor category $\mathcal{C}\# \mathbb{F}_{n}$.
\item 
If the duality operators $\{\Phi_{i}\}$ flow to an internal symmetry in the IR described by a fusion category $\cX$, then $\cX$ is a quotient of $\mathcal{C}\# \mathbb{F}_{n}$.

\end{enumerate}

\end{theorem}

\begin{corollary}\label{Cor:WeaklyIntegral}
If $A\cong \otimes_{\mathbb{Z}} M_{d}(\mathbb{C})$ is a tensor product quasi-local algebra with internal symmetry $\mathcal{C}$, then any $\cX$ emerging in the IR from an RG flow of duality operators is weakly integral.
\end{corollary}

The rest of the paper is organized as follows: in Section \ref{sec:definitions}, we lay out the formal definitions of our framework, explaining the connection between non-local operators and C$^\star$-correspondences in infinite volume limit. In Section \ref{sec:SymTFT}, we define the physical boundary subalgebra in the context of a microscopic SymTFT decomposition with symmetry category $\cC$ and a choice of charge category $\cD$ \eqref{sec:charge_cat}. In Section \ref{sec:duality_op}, we give an operator algebraic formulation of duality operators.  In Section \ref{sec:parametrization}, we parametrize the duality operators as topological interfaces between a topological boundary extension and its twisted one, which leads to the universal tensor category \eqref{sec:universal_cat}, an $\mathbb{F}_n$-graded extension of $\cC$, in the UV.

\subsection*{Acknowledgements}

The authors would like to thank Dachuan Lu, Chenqi Meng, Sahand Seifnashri, Nathanan Tantivasadakarn for useful converstaions. C.J. was supported by NSF DMS-2247202. Research at Perimeter Institute is supported in part by the Government of Canada through the Department of Innovation, Science and Economic Development and by the Province of Ontario through the Ministry of Colleges and Universities. This research was supported in part by grant NSF PHY-2309135 to the Kavli Institute for Theoretical Physics (KITP).


\section{Non-local operators in infinite volume}\label{sec:definitions}

In the operator algebraic picture of quantum field theories, a standard approach to describing theories in the infinite volume limit is in terms of the algebras of local operators. The primary mathematical object here is the quasi-local C$^\star$-algebra $A$, consisting of (the completion of) all local operators taken together. On the lattice, $A$ is typically the infinite tensor of finite dimensional matrix algebras localized at lattice sites. 

More formally, an abstract \textit{quasi-local algebra} over a metric space $(X,d)$ is a (unital) C$^\star$-algebra $A$, together with a distinguished family of unital subalgebras $A_{F}\subseteq A$, where $F\subseteq X$ ranges over all metric balls in the space (other choices of ``nice" subsets also are used depending on the context). We assume that if $F\cap G=\varnothing $, then $[A_{F},A_{G}]=0$ and $\bigcup_{F} A_{F}$ is dense in $A$. The subalgebra $A_{F}$ consists of the operators \textit{localized} in the region $F$. Though technically speaking the elements of $A$ are norm-limits of local operators (and are properly called quasi-local operators), we will slightly abuse terminology and refer to the elements of $A$ as local. We often assume some version of \textit{Haag duality}: for a ``nice" region $F$ (e.g. a ball), the local operators which commute with all the local operators localized outside $F$ are localized in $F$. This allows us to detect where something is localized by what it commutes with.

A state can be defined in this setting as a (normalized) positive linear functional $\phi: A\rightarrow \mathbbm{C}$, which in turn furnishes a Hilbert space representation of $A$ which caries $\phi$ as a vector state. The other vector states can be viewed as local perturbations of $\phi$, and in the case where $\phi$ is pure the whole Hilbert space representation is called a \textit{superselection sector} or simply sector. The existence of inequivalent sectors is the key feature of the infinite volume limit underlying the explanatory power of asymptotic reasoning with regard to macroscopic phases and phase transitions.

One important conceptual issue when working in the infinite volume limit is the problem of characterizing non-local operators. These operators map between \textit{different} Hilbert space sectors, hence cannot be naturally realized as limits of local operators in general.

From this perspective, a natural definition of a \textit{general}, a not necessarily local operator is a completely positive (cp) map $\Psi: A\rightarrow A$. This contains the local operators $a\in A$ via $\Phi_{a}(b):=aba^{*}$, and we see that $\Phi_{a}\circ \Phi_{a^{\prime}}=\Phi_{aa^{\prime}}$, giving us operator composition. We again slightly abuse terminology and call a convex combinations of this type of cp map a local operator, and refer to general cp-maps as simply operators.

Describing a state as a positive functional $\phi:A\rightarrow \mathbbm{C}$, a cp map $\Psi$ acts on $\phi$ by $\phi\circ \Psi$. In fact cp maps are the most general operators on the the state space of a C$^\star$-algebra compatible with the natural topology, convex structure, and coupling to larger systems, justifying our use of cp maps as operators.

From this perspective, it is natural to say what it means for an operator to be localized in a region $I$: $\Phi$ is localized in $I$ if for all $a\in B_{I^{c}}$, $a\Phi(1)=\Phi(a)=\Phi(1)a$. Clearly for any operator $b\in B_{I}$, $\Phi_{b}$ is localized $I$. However, in general there could be non-local operators $\Phi$ that are nevertheless localized in the finite region $I$, which is precisely the situation with string operators associated to anyons in topologically ordered quantum many-body systems. This will be the jumping off point for SymTFT and physical boundary subalgebras, which we elaborate in the net subsection.

\subsection{Sectorization of non-local operators}\label{sec:sectors}

An important feature of states in infinite volume which we have alluded to previously is \textit{sectorization}. Informally, two abstract pure states $\phi$ and $\psi$ in infinite volume  lie in the same \textit{superselection sectors}\footnote{we are using this term in the more general sense of physicists, rather than the more specialized usages in algebraic quantum field theory, or more recently in topologically ordered spin systems \cite{Naaijkens_2022}.} if they are related to each other by a quasi-local perturbation. More formally, we can describe this relationship in terms of the representation theory of $A$. 

Associated to any state $\phi$ is its GNS Hilbert space $L^{2}(A,\phi)$, which carries a bounded representation of the quasi-local algebra $A$. $L^{2}(A,\phi)$ is defined by starting with a formal state vector $\Omega_{\phi}$, and forming the vector space of formal quasi-local perturbations of this state vector $\{a\Omega_{\phi}\ |\ a\in A\}$. We then define an inner product via 

$$\langle a\Omega_{\phi}\ |\ b\Omega_{\phi}\rangle:=\phi(a^{*}b).$$

In general, this is not positive definite, so we have to quotient by the kernel and complete to obtain the Hilbert space $L^{2}(A,\phi)$, which has a natural action by $A$. Notice that $\phi$ is realized by the vector state $\Omega_{\phi}\in L^{2}(A,\phi)$, and in fact this is the universal realization of $\phi$ as a vector state.

Recall $\phi$ is a pure state if and only if $L^{2}(A,\phi)$ is an irreducible representation of $A$. From this point of view, one way to say that $\phi$ and $\psi$ lie in distinct superselection sectors is that $L^{2}(A,\phi)$ and $L^{2}(A,\psi)$ are inequivalent as representations of the algebra $A$. Conversely, two pure states are in the same superselection sector if they have (unitarily) isomorphic GNS representations, or equivalently, one state can be realized as a vector state in the GNS representation of the other.

Importantly, the above discussion shows that we can naturally lump quasi-local perturbations of a single state into one an object in a category. This leads us to pass from the set of states to the \textit{category} of Hilbert space representations of $A$ and bounded intertwiners, where we can utilize the mathematical machinery of category theory to study states and various notions of equivalence that occur in the study of quantum phases of matter. 

The reason that this story is often unfamiliar to the working physicist is that sectorization does not usually occur in finite volume: the algebra of local operators $M_{d}(\mathbbm{C})$ has a unique irreducible representation, hence all pure states lie in the same superselection sector. Indeed, the existence of distinct superselection sectors in the infinite volume limit (sectorization) lies at the heart of most rigorous explanations of macroscopic phases and emergent phenomenon, which can be witnessed with recent progress in topologically ordered phases of matter \cite{Naaijkens_2022}. If states are in different phases (whatever we might mean by this term), they are necessarily in distinct superselection sectors though the converse is not usually the case. Nevertheless, this principle tells us that whenever we want to mathematically characterize phases, working with sectors (and hence working in the category of Hilbert space representations of the quasi-local algebra).

The key observation for us is that sectorization also occurs with \textit{non-local operators}. Similarly to the case of states, we can consider quasi-local perturbations of a cp-map via a GNS type construction. However, the role of Hilbert space representation of $A$ that we use in this picture must be generalized to a \textit{correspondence} between C$^\star$-algebras. 

\begin{defn}
Let $A$ and $B$ be C$^\star$-algebra. A \textit{correspondence} is a linear $A$-$B$ bimodule $X$, together with a map $X\times X\rightarrow B$, written $(x,y)\mapsto \langle x\ |\ y\rangle\in B$, satisfying

\begin{enumerate}
\item 
$\langle \cdot\ |\ \cdot\rangle$ is linear in the second variable and conjugate linear in the first variable.
\item 
$\langle x\ |\ y\rangle^{*}=\langle y\ |\ x\rangle$.
\item 
$\langle x\ |\ x\rangle\ge 0$ in $B$, and $\langle x\ |x\rangle =0$ implies $x=0$.
\item 
The norm $\| x\|=||\langle x\ |x\ \rangle||^{\frac{1}{2}}$ is a Banach norm.
\item 
$\langle x\ |\ y\triangleleft b\rangle=\langle x\ |\ y\rangle b$.
\item 
$\langle a\triangleright x\ |\ y\rangle=\langle x\ |\ (a^{*})\triangleright b\rangle$
\end{enumerate}
\end{defn}

Notice that if $B=\mathbbm{C}$, then $X$ with $\langle \cdot\ |\ \cdot\ \rangle$ becomes a Hilbert space, and the left action of $A$ is simply a Hilbert space representation. Let $X$ be an $A$-$B$ correspondence, and $x\in X$. Then the matrix coefficient associated to $x$ defines a cp map $\Phi_{x}:A\rightarrow B$ by

$$\Phi_{x}(a)=\langle x\ |\ a\triangleright x\rangle\,.$$

Conversely, starting from a cp map $\Phi: A\rightarrow B$, we start with a ``state vector" $\Omega_{\Phi}$, but now perturb on the left by elements of $A$ and on the right by elements of $B$

$$\text{span}\{a\Omega_{\Phi}b\, :\ a\in A, b\in B\}\,.$$

Then we define a $B$-valued inner product $$\langle a_{1}\Omega_{\Phi}b_{1}\ |\ a_{2}\Omega_{\Phi}b_{2}\rangle:=b^{*}_{1}\Phi(a^{*}_{1}a_{2})b_{2}.$$

Again, we quotient by the kernel and complete to obtain an $A$-$B$ correspondence, which we call $L^{2}(\Phi)$. This is the universal correspondence containing $\Phi$ as a ``vector state".

For the same reason as for states, it is extremely useful to consider non-local operators together with their quasi-local perturbations in a single object, namely a correspondence. In particular, we can access categorical structure. Indeed, the collection of $A$-$B$ correspondences forms a C$^\star$-category, where morphisms between $A$-$B$ correspondences $X$ and $Y$ are linear bimodule maps $f:X\rightarrow Y$ such that $\langle f(x)|y\rangle_{Y}=\langle x|f^{*}(y)\rangle_{X} $ for some $f^{*}:Y\rightarrow X$. Such maps are automatically bounded, and composition is the usual (for details, see \cite{Chen_2022}).

Operators, unlike states, have a natural composition operators, which is the crucial part of their existence. This is reflected at the correspondence level through the \textit{tensor product} of correspondences. If $X$ is an $A$-$B$ correspondence and $Y$ is a $B$-$C$ correspondence, then $X\boxtimes Y$ is an $A$-$C$ correspondence, defined with an inner product on simple tensors

$$\langle x_{1}\otimes y_{1}\ |\ x_{2}\otimes y_{2}\rangle_{X\boxtimes Y}:=\langle y_{1}\ |\ \langle x_{1}\ | x_{2}\rangle_{X}\triangleright y_{2}\rangle_{Y}.$$

As usual we quotient by the kernel and complete, then equip the resulting space with the obvious left $A$-action and right $C$-action. If $x\in X$ and $y\in Y$, and $\Phi_{x}, \Phi_{y}$ are the associated vector state cp maps, then $\Phi_{x}\circ \Phi_{y}$ is a vector state in $X\boxtimes Y$ via the state vector $x\otimes y$. 

The collection of C$^\star$-algebras, correspondences, and intertwiners assembles into a C$^\star$ 2-category \cite{Chen_2022}. For a single algebra $A$, the collection of all $A$-$A$ correspondences is a C$^\star$-tensor category, which we denote $\text{Bim}(A)$. For a quasi-local algebra, this C$^\star$-tensor category is the ``sectorization" of non-local operators.

\subsection{Microscopic SymTFT picture}\label{sec:SymTFT}

Fusion category symmetries of quantum field theories have emerged as an important tool in understanding quantum field theories and quantum many-body systems, particularly in \spd{1} where the story is better understood. In this paper, we will focus on the lattice setting. The usual picture for fusion categorical symmetry utilizes an algebra of ``symmetry operators", encoded either by matrix product operators (MPOs) or (weak) Hopf C$^\star$-algebra acting on the local Hilbert spaces \cite{meng2025noninvertiblesptsonsiterealization,inamura202611dsptphasesfusion,molnar2022matrixproductoperatoralgebras}. An alternative picture has emerged which provides powerful tools for analyzing the structure of symmetric phases (both gapped and gappless) called SymTFT \cite{Kong_2017,freed2024topologicalsymmetryquantumfield,evans2025operatoralgebraicapproachfusion,Bhardwaj_2025}.

In the SymTFT picture for a spin system with local Hamiltonian $H=\sum H_{I}$, we have a decomposition of the system into a one dimension higher TQFT (called the SymTFT) sandwiched between a topological boundary and a physical boundary, as in the following picture:

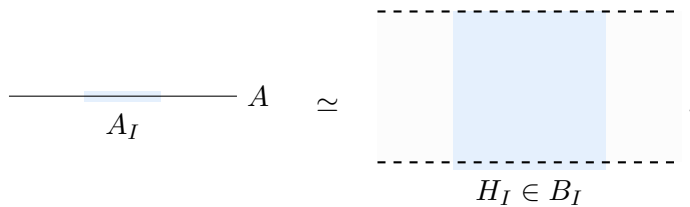
\[
\diagram{1}{
\fill[YaleLightBlue] (1,-.07) rectangle (2,.07);
\draw (0,0)--(3,0)node[right]{$A$};
\node at (1.5,-.4) {$A_I$};
}\quad  \simeq \quad  
\diagram{2}{
\fill[YaleLightGrey] (0,-.5) rectangle  (2,.5);
\fill[YaleLightBlue] (.5,-.55) rectangle  (1.5,.5);
\node at (1,-.7) {$H_I \in B_I$};
\draw[trivial](0,-.5)--(2,-.5);
\draw[trivial](0,.5)--(2,.5);
}\,.
\]

The physical boundary (also sometimes called the dynamical boundary) is usually very complicated, and should contain the local terms $H_{I}$ of the Hamiltonian. The topological boundary, in contrast, is simple and witnesses the categorical symmetry via topological defects. A key feature of this story is that the symmetry operators are secondary, and the TQFT itself takes center stage. This allows for the direct application of TQFT ideas to the study of symmetric phases.

From the operator algebra perspective, it is natural to try to access a SymTFT decomposition by idenitifying the subalgebra $B\subseteq A$ of local operators localized near the physical boundary, and then trying to recover the rest of the SymTFT structure from this. This leads to a natural question: when can we realize an abstract quasi-local algebra the local operators at a physical boundary of a TQFT? This problem was addressed in \cite{evans2025operatoralgebraicapproachfusion}.

The key idea is that the non-local string operators that create anyons in the bulk can be pushed into the boundary to define \textit{non-local} operators on the boundary theory. Due to the topological nature of the bulk string operators, we expect that after local perturbations (given by post composing, so ``right" perturbations), we can obtain operators that are \textit{localized in any other (sufficiently large) region we like}. In particular, the correspondence associated to a string operator should contain operators localized in any sufficiently large region, and should be generated as a right modules by the pieces of the string operator localized in any large enough interval.

To formalize this at the level of correspondences, let's focus on the case where space is discrete and one-dimensional, i.e. $\mathbbm{Z}$. We assume that the local subalgebras $A_{I}$ are finite dimensional C$^\star$-algebras assigned to intervals $I$, and that we have Haag duality for intervals.

For any $A$-$A$ correspondence $X$, set $X_{I}:=\{x\in X\ : ax=xa\ \text{for all}\ a\in A_{I^{c}}\}$. The vectors $x\in X_{I}$ are precisely the elements that correspond to operators localized in $I$, i.e. the associated cp maps $\Phi_{x}:A\rightarrow A$ satisfy $\Phi_{x}(a)=a\Phi_{x}(1)=\Phi_{x}(1)a$ for all $a\in A_{I^{c}}$. A DHR bimodule will essentially be any correspondence generated under post-composition with quasi-local perturbations by operators localized \textit{anywhere}. Because of the finite scale of the lattice, we will only assume that we can be localized in any sufficiently large interval. We have the following formal definition.

\begin{defn} 
(\textbf{DHR bimodules}). An $A$-$A$ correspondence $X$ is a \textit{DHR}-bimodule if for there is some $R\ge 0$ so that for any interval with $|I|\ge R$, $\text{dim}(X_{I})<\infty$ and $X=X_{I}A$.
\end{defn}

This definition looks slightly different than previous definitions in the literature \cite{jones2024dhrbimodulesquasilocalalgebras, evans2025operatoralgebraicapproachfusion}. We say an $A$-$B$ correspondence $X$ is right finite if there exists a finite set $\{b_{i}\}\subseteq X$ such that for all $x\in X$, 

$$x=\sum_{i} b_{i}\langle b_{i}|x\rangle\,.$$

We call $\{b_{i}\}$ a finite projective basis. Now we make contact with the usual definition of DHR bimodule in the literature.

\begin{theorem}
Suppose $A$ is a locally finite-dimensional net of algebras over $\mathbbm{Z}$ satisfying weak algebraic Haag duality. An $A$-$A$ correspondence $X$ is a DHR bimodule if and only if there exists an $S$ such that for all intervals $|I|\ge S$, there is a finite projective basis contained in $X_{I}$.
\end{theorem}

\begin{proof}
Suppose $|I|\ge R$, so that $X_{I}A=X$. By weak Haag duality, there is some $T$ such that $A^{\prime}_{I^{c}}\cap A\subseteq A_{I^{+T}}$. Thus if we take $X_{I}A_{I^{+T}}\subseteq X_{I^{+T}}$, restricting the inner product gives this subspace the structure of a right Hilbert module over the finite dimensional C$^\star$-algebra $A_{I^{+T}}$. But Hilbert modules over finite dimensional C$^\star$-algebras always admit projective bases, and thus we can find a $\{b_{i}\}\subseteq X_{I^{+T}}$ with $$\sum_{i} b_{i}\langle b_{i} | x\rangle=x$$ 

\noindent for all $x\in X_{I}A_{I^{+T}}.$ But by the DHR condition then $X=X_{I}A\subseteq X_{I}A_{I^{+T}}A=X$. Thus any $x\in X$ can be written as $\sum _{i} x_{i}a_{i}$, where $x_{i}\in X_{I}A_{I^{+T}}$ and $a_{i}\in A$.

Thus

\begin{align*}
\sum_{i} b_{i}\langle b_{i} | x \rangle&=\sum_{i,j} b_{i}\langle b_{i} | x_{j}a_{j}\rangle\\
&=\sum_{j} \left(\sum_{i} b_{i}\langle b_{i} | x_{j}\rangle\right)a_{j}\\
&=\sum_{j}x_{j}a_{i}=x\,.
\end{align*}
Hence we have a projective basis localized in $I^{+T}$, and thus setting $S=R+T$, we have the desired result.

\end{proof}

Now, we connect abstract DHR bimodules to the bulk SymTFT picture. 

\begin{theorem} [\cite{jones2024dhrbimodulesquasilocalalgebras}]. 
The C$^\star$-tensor category of $\text{DHR}(A)$ consisting of DHR bimodules has a natural unitary braiding.
\end{theorem}

We can now describe the SymTFT picture.

\begin{defn}
Let $A$ be a quasi-local algebra over $\mathbbm{Z}$ which is locally finite dimensional and satisfies weak algebraic Haag duality. A unital subalgebra $B\subseteq A$ is called a physical boundary subalgebra if

\begin{enumerate}
\item 
$\bigcup_{I} (B\cap A_{I})$ is norm dense in $B$.
\item 
There exists a conditional expectation $E:A\rightarrow B$ such that $(B\subseteq A, E)$ is a Lagrangian Q-system in $\text{DHR}(B)$.
\end{enumerate}
\end{defn}

Some immediate consequences of the definition:

\begin{itemize}
    \item 
    $E:A\rightarrow B$ is locality preserving, i.e. there is some $R$ such that $E(A_{I})\subseteq B_{I^{+R}}=B\cap A_{I^{+R}}$.
    \item 
    The inclusion is irreducible, i.e. $B^{\prime}\cap A=\mathbbm{C}$.
    \item 
    $\text{DHR}(A)$ is trivial.
    \item 
    There is a fusion category $\mathcal{C}$ of $A$ modules internal to $\text{DHR}(B)$, realized concretely as a class of correspondences over $A$, such that $\text{DHR}(B)\cong \mathcal{Z}(\mathcal{C})$ \cite{evans2025operatoralgebraicapproachfusion,hataishi2025structuredhrbimodulesabstract}
    \item 
    The fusion hypergroup of $\mathcal{C}$ acts by unital completely positive maps on $A$, such that the invariant quasi-local operators are precisely $B$.
\end{itemize}

The fusion category $\mathcal{C}$ above is called the \textit{symmetry category}. Notice that from this perspective, the symmetry category $\mathcal{C}$ is secondary, and the symmetric operators are what we focus on. The symmetric operators can arise however you wish: for example, as the local operators that commute with an MPO or (weak)-Hopf algebra action. The category $\mathcal{C}$ and its action by quantum channels (ucp maps) are \textit{canonically defined} by the symmetric subalgebra. 

\subsection{Charge categories}\label{sec:charge_cat}
Focusing on the physical boundary algebra $B$ itself, we can ask a slightly different question: which abstract quasi-local algebras over $\mathbbm{Z}$ are suitable to be physical boundaries of a TQFT in the first place? An obvious answer is that $\text{DHR}(B)$ should be the Drinfeld center of a fusion category. In \cite{hataishi2025structuredhrbimodulesabstract}, it was shown that under some reasonable assumptions $\text{DHR}(B)$ is automatically the Drinfeld center of a category called the \textit{charge category} $\mathcal{D}$, which is generally different from (or ``dual to) the symmetry category $\mathcal{C}$, the latter of which only arises when $B$ is embedded as a physical boundary subalgebra of a quasi-lcoal algebra $A$.

To describe this category, consider the truncated quasi-local algebra $B_{-}$, which is the C$^\star$-subalgebra of $B$ generated by $B_{I}$ with $I\le 0$. We are interested in non-local operators (i.e. cp maps) defined only on $B_{-}$ that are localized around $0$. Following the translation of this idea into bimodules, we have the following definition.

\begin{defn}
A $B_{-}$-$B_{-}$ correspondence $X$ is a \textit{charge bimodule} if there exists some $I=[-k,0]$ such that $X_{I}B=X$. We denote the category of charge bimodules by $\text{DHR}_{-}(B)$.
\end{defn}

$\text{DHR}_{-}(B)$ is a C$^\star$-tensor category with the relative tensor product, however it does not have a braiding. From \cite{hataishi2025structuredhrbimodulesabstract,jones2025holographybulkboundarylocaltopological}, we do in general have a braided tensor functor $\text{DHR}(B)\rightarrow \mathcal{Z}(\text{DHR}(B_{-}))$. Under some reasonable assumptions, $\text{DHR}_{-}(B)$ is a fusion category, and the functor described above is an equivalence \cite{hataishi2025structuredhrbimodulesabstract,jones2025holographybulkboundarylocaltopological}. In particular, for (multi)-fusion spin chains (see the next section), this is a braided equivalence \cite{hataishi2025structuredhrbimodulesabstract}.

The main class of abstract physical boundary algebras  that we consider are fusion spin chains, built from a fusion category $\mathcal{D}$, and a strongly tensor generating object $X$. In this case, the charge category is $\mathcal{D}$ itself. Note that in the categorical symmetry context, the fusion spin chains in question arise as the \textit{symmetric subalgebras} of a larger quasi-local algebra $A$ with symmetry category $\mathcal{C}$ which in general is distinct from $\mathcal{D}$, but is always Morita equivalent to it by our above discussion.

\begin{example}[Q-system models]

We briefly review the \spd{1} lattice model for topological phases enriched with fusion category symmetries discussed in \cite{meng2025noninvertiblesptsonsiterealization}. This specific construction provides a convenient way to construct symmetric local operators on an $\sH$-symmetric quantum spin chain, where $\sH$ is a Hopf C$^\star$-algebra. We start with the Hopf C$^\star$-algebra $\sH$ as the local version of the fusion category symmetry and define the charge category as the representation category of the Hopf C$^\star$-algebra, isomorphic to $\cC_{\Hilb_f}^\vee:=\text{Fun}_\cC(\Hilb_f,\Hilb_f)$, where $f$ is a unitary fiber functor $f: \cC \rightarrow \Hilb$. A generic lattice model with fusion category symmetry $(\cC, f)$ consists of: local Hilbert spaces that carry representations of the Hopf algebra and local interactions given by Hermitian intertwiners between representations on the local Hilbert space \cite{Inamura_2022}, which directly generalizes the construction with finite group symmetry in \cite{Lan_2024}. Note that for different fiber functors $f, h$ of $\cC$, the associated dual categories are not the same $\cC_{\Hilb_f}^\vee \ncong \cC_{\Hilb_h}^\vee$. Thus the choice of the charge category is not unique. Examples include the isocategorical groups, $\Rep^\dagger (D_{4n})$ for $n \in \Z, n >3$, $\Rep^\dagger (A \rtimes \Z_2) \times A$ for finite abelian group $A$ \cite{meng2025noninvertiblesptsonsiterealization}. We will choose the charge category as the one associated with the trivial fiber functor $\fgt: \cC \rightarrow \Hilb$, and call it the reference charge category. 

    Given a symmetry category $\cC = \Rep^\dagger(\sH^*)$, we choose the reference charge category $\cD:= \cC_{\Hilb_{\fgt}}^\vee$ in which objects are symmetry charges (localized excitations on the physical boundary) and morphisms are interactions. We use fusion channels to define the local interactions. A local $\cC$-symmetric tensor in region $I$ is
\[
\begin{tikzpicture}[baseline={($(current bounding box)$)}]
        \draw[string] (0,-1) -- node[right]{\small $X_1$} (0,-.3);
        \draw[string] (0,.3) -- node[right]{\small$Y_1$} (0,1);
        \filldraw[fill=white] (-.3,-.3) rectangle node{$f$} (2.3,.3);
        \draw[string] (.6,-1) -- node[right]{\small$X_2$} (.6,-.3);
        \draw[string] (.6,.3) -- node[right]{\small$Y_2$} (.6,1);
        \node at (1.5,.6) {$\cdots$};
        \node at (1.5,-.6) {$\cdots$};
        \draw[string] (2,-1) -- node[right]{\small$X_n$} (2,-.3);
        \draw[string] (2,.3) -- node[right]{\small$Y_m$} (2,1);
        \node at (1,-1.4) {$X^{\ot I}$};
\end{tikzpicture} \in A_I := \End_{\cD}(X^{\ot I})
\]
For $I \subseteq J$, we define the natural inclusion
\begin{align*}
	A_I \hookrightarrow A_J, \quad f \mapsto \one_X^{\ot J < I } \ot f\ot \one_X^{\ot J >I}
\end{align*}
which is graphically represented as
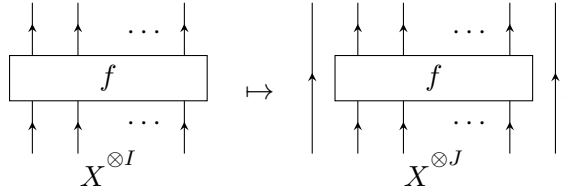
\[
\begin{tikzpicture}[baseline={($(current bounding box)$)}]
        \draw[string] (0,-1) --  (0,-.3);
        \draw[string] (0,.3) --  (0,1);
        \filldraw[fill=white] (-.3,-.3) rectangle node{$f$} (2.3,.3);
        \draw[string] (.6,-1) --  (.6,-.3);
        \draw[string] (.6,.3) --  (.6,1);
        \node at (1.5,.6) {$\cdots$};
        \node at (1,-1.2) {$X^{\ot I}$};
        \node at (1.5,-.6) {$\cdots$};
        \draw[string] (2,-1) --  (2,-.3);
        \draw[string] (2,.3) --  (2,1);
\end{tikzpicture} \quad \mapsto \quad 
\begin{tikzpicture}[baseline={($(current bounding box)$)}]
		\draw[string] (-.6, -1)--(-.6,1);
		\draw[string] (2.6, -1)--(2.6,1);
        \draw[string] (0,-1) --  (0,-.3);
        \draw[string] (0,.3) --  (0,1);
        \filldraw[fill=white] (-.3,-.3) rectangle node{$f$} (2.3,.3);
        \draw[string] (.6,-1) --  (.6,-.3);
        \draw[string] (.6,.3) --  (.6,1);
        \node at (1.5,.6) {$\cdots$};
        \node at (1,-1.2) {$X^{\ot J}$};
        \node at (1.5,-.6) {$\cdots$};
        \draw[string] (2,-1) --  (2,-.3);
        \draw[string] (2,.3) --  (2,1);
\end{tikzpicture} \,.
\]
Then define $$A:=\text{colim}_{I} A_{I}.$$ Identifying each $A_{I}$ with its image in $A$ yields an abstract spin system over $\mathbbm{Z}$. We will give the detailed basis decomposition along with the duality operators in Section \ref{sec:basis}. 

For now, we quote the results in \cite{meng2025noninvertiblesptsonsiterealization} and provide a convenient realization of symmetric local operators.

As discussed above, given a Hopf C$^\star$-algebra $\sH$ describing the local version of a fusion category symmetry, the local Hilbert space of an $\sH$-symmetric quantum spin chain is $\mathrm{Irr}(\Rep^\dagger(\sH))$. This is a large-enough local Hilbert space, in which different $\sH$-symmetric phases (SPTs or SSBs) are embedded. $\sH$-symmetric phases are in one-to-one correspondence with the Morita classes of unitary separable Frobenius algebras (Q-systems) in $\cC_{\Hilb_f}^\vee$. On an infinite chain, the Q-system is labeled by $(\Z,\,\sH,\, A,\,1-m^\dagger m)$, where the each site of the spin chain is indexed by $i \in \Z$, $\sH$ is the Hopf C$^\star$-algebra describing the symmetry of the spin chain, $\Psi(\{i\})=A$ is the local Hilbert space at site $i$ and $\Phi(\{i,i+1\}) = 1-m^\dagger m$ gives the nearest-neighbor interaction. In short, the total Hilbert space is given by $\cH = \bigotimes_{i \in \Z}A$ and the Hamiltonian is
\[
H = \sum_{i \in \Z}1-(m^\dagger m)_{i,i+1} = \sum_{i \in \Z}1- \,
\diagram{.5}{

        \draw (.1,0) -- node[above right]{\tiny$\;m$} node[below left]{\tiny$A$}(.6,.5) -- node[below right]{\tiny$A$}(1.1,0);
        \draw (.1,2) -- node[below right]{\tiny$\;m^\dagger$} node[above left]{\tiny$A$}(.6,1.5) -- node[above right]{\tiny$A$}(1.1,2);
        \draw (.6,.5) -- (.6,1.5);
    }\,.
\]
We summarize the notations in the table below
 \begin{center}
{\renewcommand{\arraystretch}{2}
\begin{tabular}{|>{\RaggedRight\arraybackslash}p{5cm}|
                >{\RaggedRight\arraybackslash}p{6cm}|}
  \hline
  (dual) Hopf C$^\star $-algebra & ($\mathsf{H^*}$) $\mathsf{H}$\\ 
  \hline
  Symmetry category $\mathcal{C}$ & $ \Rep^\dagger(\mathsf{H^*})$ \\ 
  \hline
  Charge category $\cD:=\mathcal{C}_{\Hilb_f}^\vee$ & $\Rep^\dagger(\mathsf{H})$ \\ 
  \hline
  Commuting projector fixed point model $(\Z,\,\sH,\,A,\,1-m^\dagger m)$ &  $A$ is Q-system in $\mathcal{C}_{\Hilb_f}^\vee$,
  local Hilbert space $\mathcal{H}_i = A$ $\forall \, i \in \Z$, Hamiltonian $H = \sum_i 1-(m^\dagger m)_{i,i+1}$.
 \\
  \hline
\end{tabular} 
}
\end{center}
Thus given a charge category, we can construct a \spd{1} lattice model realizing all $\sH$-symmetric phases via the projection $P_k: \mathrm{Irr}(\Rep^\dagger(\sH)) \rightarrow A_k$. The Hamiltonian is written as
\[
H = - \sum_k \sum_{i=1}^L \lambda_k \left( m_k^{i,i+1} P_k^i P_k^{i+1} \right)^\dagger \left( m_k^{i,i+1} P_k^i P_k^{i+1} \right)\,
\]
for all $A_k$ in $\mathcal{C}_{\Hilb_f}^\vee$ with the associated tuning parameter $\lambda_k$.

When $\cC=\Hilb_{\Z_2}$, $\cC=\Rep^\dagger(\Z_2)$, we have $A_0 = \Fun(\Z_2/\Z_2)$ and $A_1=\Fun(\Z_2/e)$, giving the generators of symmetric local operators $\{Z_iZ_{i+1},\,X_i\}_{i \in \Z}$.

\end{example}

\section{Duality operators}\label{sec:duality_op}

We will treat the physical boundary algebra $B$ as fundamental, and consider the different topological boundaries of our SymTFT as variable. Given a fixed physical boundary algebra $B$, the possible topological boundaries in the SymTFT decomposition correspond to Lagrangian algebras $A\in \text{DHR}(B)$.

\begin{defn}
Let $B\subseteq A$ be a SymTFT decomposition with fusion category symmetry $\mathcal{C}$.  A \textit{duality operator} consists of a completely positive map $\Psi: A\rightarrow A$ such that:

\begin{enumerate}
\item 
There exists an $R\ge 0$ such that $\Psi(A_{F})\subseteq (A)_{F^{+R}}$

\item 
$\Psi$ restricts to a $*$-automorphisms from $B$ to $B$.

\end{enumerate}
\end{defn}

We note that one consequence of our assumption, is that $B$ is in the multiplicative domain of the channel $\Psi$, so that in particular

$\Psi(b_{1}a b_{2})=\Psi(b_1)\Psi(a)\Psi(b_2)$.

We also note that $\Psi$ is automatically unital. If we take a duality channel $\Psi:A_{1}\rightarrow A_{2}$, $\Psi|_{B}$ is a bounded spread isomorphism, sometimes called a quantum cellular automata (QCA).

\begin{quest}\label{quest1}
Given a bounded spread automorphism $\alpha:B\rightarrow B$, how do we parameterize extensions to duality channels $\Psi: A_{1}\rightarrow A_{2}$ such that $\Psi|_{B}=\alpha$?
\end{quest}

The version of this question which asks when we can extend a bounded spread isomorphism on $B$ to an honest (invertible) QCA on $A$ was answered in \cite{jones2025quantumcellularautomatacategorical}.

Another important reason for considering duality operators is their role in the context of generalized symmetry. It was argued in \cite{Seiberg_2024,Li_2023} that the canonical example of a duality operators, the Kramers-Wannier operator (see the next section for a full explanation), implements a form of generalized fusion category symmetry that goes beyond the usual internal story. This is witnessed by the fact that the Kramers-Wannier operator \textit{almost} has the fusion rules of the Ising fusion category, but ``mixes with" translation. The idea is that under any RG flow which is constrained by the internal $\mathbbm{Z}_{2}$ symmetry, this will give rise to an honest internal symmetry in the IR. This concept is sometimes called \textit{emmanent symmetry} in the literature \cite{Cheng_2023}.

Part of our motivation is to systematically study this type of symmetry using our analysis of duality operators.




\subsection{Bimodule channels and topological boundary defects}

Suppose we have a physical boundary algebra $B$ and two topological boundary extensions $B\subseteq A_{1}$ and $B\subseteq A_{2}$. We want to capture topological defects \textit{between} these boundary conditions in terms of quantum channels between $A_{1}$ and $A_{2}$. In the SymTFT picture, this is obtained by ``sweeping out" the defect. But this procedure acts trivially on the physical boundary. In pictures, this looks as follows.

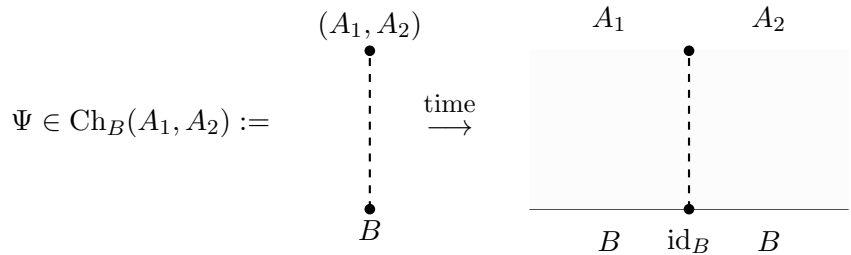
\[\Psi \in \text{Ch}_B(A_1,A_2):=\quad
\diagram{1.4}{
\draw (0,-.5)--(3,-.5);
\draw (0,1)--(3,1);
\fill[YaleLightGrey] (0,-.5) rectangle  (3,1);
\draw[trivial] (1.5,-.5)--(1.5,1);
\draw[trivial] (-1.5,-.5)node[below]{$B$}--(-1.5,1)node[above]{$(A_1,A_2)$};
\node at (.75,-.8) {$B$};
\node at (2.25,-.8) {$B$};
\node at (.75,1.3) {$A_1$};
\node at (2.25,1.3) {$A_2$};
\node at (1.5,-.8) {$\id_B$};
\fill(1.5,-.5) circle [radius = .05];
\fill(1.5,1) circle [radius = .05];
\fill(-1.5,-.5) circle [radius = .05];
\fill(-1.5,1) circle [radius = .05];
\node at (-.75,.25){$\longrightarrow$};
\node at (-.75,.5){\small$\text{time}$};
}
\]

To capture this idea in operator algebra language, we have the following definition.

\begin{defn} Let $B$ be an abstract spin system and $B\subseteq A_{1}$ and $B\subseteq A_{2}$ local extensions. A bimodule channel between these extensions is a completely positive map $\Psi:A_{1}\rightarrow A_{2}$ such that $\Psi|_{B}=\text{Id}_{B}$. We denote the set of bimodule channels $\text{Ch}_{B}(A_{1},A_{2})$.
\end{defn}

We note that as before, this implies $\Psi$ is $B$-$B$ bimodular, and automatically unital. The collection of topological boundary extensions and bimodule channels extends to a category, defined as follows.

\begin{defn} Fix a physical boundary algebra $B$, and define category $\text{T}\partial(B)$ as follows: 

\begin{itemize} 
\item Objects are all topological boundary extensions of $B$.

\item The set of morphisms between extensions $A_{1}$ and $A_{2}$ is given by the set $\text{Ch}_{B}(A_{1},A_{2})$, and composition is just composition of channels.
\end{itemize}

\end{defn}

We note that $\text{T}\partial(B)$ can be viewed as enriched over the appropriate category of convex spaces. Indeed, there is an obvious convex structure on $\text{Ch}_{B}(A_{1},A_{2})$, and this composition is ``bilinear".

Our goal is to show that this formally defined category of cp maps ``matches" with the intuitive picture of topological defects between topological boundary conditions.

Notice that a bimodule channel is an actual morphism $\Psi: A_{1}\rightarrow A_2$ \textit{in the linear category} $\text{DHR}(B)$. The question is, which abstract morphisms $\psi: A_{1}\rightarrow A_{2}$ in $\text{DHR}(B)$ correspond to bimodule channels?

Let $L_{1}, L_{2}$ be commutative Q-systems in a unitary braided tensor category. We recall the \textit{convolution algebra} associated to these, as in \cite{10.21468/SciPostPhys.15.3.076}, see also \cite{BJ21,LCLJ}. Consider the finite dimensional vector space $\text{Hom}_{\mathcal{C}}(L_{1}, L_{2})$, with convolution product

$$f* g:=m\circ (f\otimes g)\circ m^{\dagger}\in\text{Hom}_{\mathcal{C}}(L_1, L_2).$$

This is a commutative C$^\star$-algebra with $*$-operation denote by $\#$ and defined as $f^{\#}:=\overline{f}$:
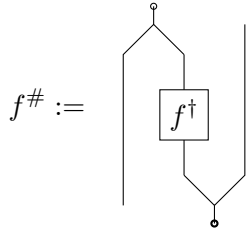
\[
f^{\#} := \quad
\diagram{.8}{
\draw(0,-1)--(0,1);
\node[draw = black, fill = white] at (0,0) {$f^\dagger$};
\draw(0,-1)--(.5,-1.5)--(1,-1)--(1,1.5);
\draw(0,1)--(-.5,1.5)--(-1,1)--(-1,-1.5);
\filldraw[fill=white,thick] (.5,-1.8) circle [radius=.05];
\draw(.5,-1.5)--(.5,-1.8);
\filldraw[fill=white] (-.5,1.8) circle [radius=.05];
\draw(-.5,1.5)--(-.5,1.8);
}
\]
We denote this C$^\star$-algebra as $H(L_{1},L_{2})$.

\begin{prop}
Let $B$ be a physical boundary algebra with topological boundary extensions $B\subseteq A_{1}, A_{2}$ and let $\Psi:A_{1}\rightarrow A_{2}$ be a $B$-$B$ bimodule intertwiner. Then $\Psi$ is completely positive if and only if $\Psi$ is a positive element in the C$^\star$-algebra $H(L_{1},L_{2})$. 
\end{prop}

\begin{proof}
Let $\mathcal{M}_{1}$ and $\mathcal{M}_{2}$ be the unitary categories of $B$-$A_{1}$ and $B$-$A_{2}$ correspondences that forget to $\text{DHR}(B)$ as $B$-$B$ correspondences. Then $B$-$B$ bimodular cp-maps $\Psi:A_{1}\rightarrow A_{2}$ are in bijective correspondence with cp-multipliers between $\mathcal{M}_{1}$ and $\mathcal{M}_{2}$ in the sense of \cite{MR3687214}, which by \cite{HP23, LCLJ, 10.1063/5.0071215} correspond to positive elements in the C$^\star$-algebra $H(L_{1},L_{2})$. 
\end{proof}

In the above statement, we gave a criteria for determining if a morphism is cp, whereas the morphisms in $\text{T}\partial(B)$ are unital cp maps. However, every non-zero $B$-bimodular cp-multiplier $\Psi$ is almost unital, in the sense that $\Psi(1_{A_1})=\lambda 1_{A_2}$ for some strictly positive scalar $\lambda$. Indeed, the irreducibility of the inclusions ($B^{\prime}\cap A_{i})=\mathbbm{C}1_{A_i}$ implies both $A_{1}$ and $A_{2}$ as $B$-$B$ bimodules contain $B$ with multiplicity $1$. Since $\Psi$ is $B$-bimodular, this implies $\Psi(1_{A_{1}})=\lambda 1_{A_{2}}$ for some scalar $\lambda$. To see this is non-zero, note that since $\Psi$ is cp, $\Psi(1_{A_{1}})=p\in A_{2}$ is a non-zero positive element of the C$^\star$-algebra $A_{2}$ if and only if $\Psi$ is non-zero.

Now, since finite-dimensional commutative C$^\star$-algebras are isomorphic to $\mathbbm{C}^{n}$, there exists a basis of orthogonal minimal projections $p_{i} \in H(L_{1}, L_{2})$ such that 

$$p_{i}*p_{j}=\delta_{i,j} p_{i}$$
$$p^{\#}_{i}=p_i$$
$$i\circ i^{\dagger}=\sum p_{i}$$

The equation follows from that fact that $i\circ i^{\dagger}$ is the unit for $H(L_{1},L_{2})$.

For each $i$, there is some non-zero, positive scalar $\lambda_i$ such that $p_{i}(1)=\lambda_{i}1$. We then define the channel

$$\Psi_{i}:=\frac{p_{i}}{\lambda_i}.$$

The next proposition shows that the convex space $\text{Ch}_{B}(A_{1},A_{2})$ is actually a simplex, meaning that there are finitely many extreme points and every element is a uniquely specifiable as a convex combination of extreme points.

\begin{theorem}
    If $B\subseteq A_{1}$ and $B\subseteq A_{2}$ are local extensions of an abstract spin system, then $\text{Ch}_{B}(A_{1},A_{2})$ is a simplex with extreme points given by $\{\Psi_{i}\}$ described above.
\end{theorem}

\begin{proof}
Since the $p_{i}$ are bases for the whole convolution C$^\star$-algebra $H(A_{1},A_{2})$, the B-bimodular ucp maps $\Psi:A_{1}\rightarrow A_{2}$ can be uniquely written written as $\Psi=\sum_{i} \beta_{i} \Psi_{i}$, but since the $\Psi_{i}$ are unital, $\Psi(1)=1$ implies $\sum_{i} \beta_{i}=1$, hence $\Psi$ is a convex combination.
\end{proof}


Now we will demonstrate how to upgrade results from \cite{BJ21}, to allow us to completely understand cp maps between Lagrangian algebras. Recall the canonical correspondence between Lagrangian algebras $L\in \mathcal{Z}(\mathcal{D})$ and indecomposable right $\mathcal{D}$-module categories. Given a Lagrangian algebra, we write the corresponding $\mathcal{D}$-module category as $\mathcal{M}_{L}$.

The collection of right module categories of $\mathcal{D}$ assembles into a unitary 2-category, whose 1-morphisms are $\star$-module functors, and morphisms are natural transformations. We will turn this into a unitary multi-fusion category as follows: pick a representative of each indecomposable module category $\mathcal{M}_{i}$ and consider the right unitary module category $\mathcal{M}:=\bigoplus_{i} \mathcal{M}_{i}$. For convenience we set $\mathcal{M}_{1}=\mathcal{D}$ as a right $\mathcal{D}$-module category. 

Then define the indecomposable unitary multi-fusion category 

$$\widetilde{\cD}:=\text{End}_{\mathcal{D}}(\mathcal{M})\,.$$

$\widetilde{\mathcal{D}}$ is an indecomposable unitary multifusion category, where the blocks are indexed by the module categories and the blocks are given by $$\mathcal{D}_{ij}:=\text{End}_{\mathcal{D}}(\mathcal{M}_{j}, \mathcal{M}_{i})\,.$$

The $\mathcal{D}_{ij}$ are invertible $\mathcal{D}_{ii}-\mathcal{D}_{jj}$ bimodule categories. Furthermore, We have 

$$\mathcal{Z}(\mathcal{D})\cong \mathcal{Z}(\widetilde{\mathcal{D}})\cong \mathcal{Z}(\mathcal{D}_{ii})\,.$$

Now let $F_{i}:\mathcal{Z}(\mathcal{D})\rightarrow \mathcal{Z}(\mathcal{D}_{ii})$ be the composition of the equivalence $\mathcal{Z}(\mathcal{D})=\mathcal{Z}(\mathcal{D}_{11})\cong \mathcal{Z}(\mathcal{D}_{ii})$ with the forgetful functor to $\mathcal{D}_{ii}$. Let $I_{i}:\mathcal{D}_{ii}\rightarrow \mathcal{Z}(\mathcal{D})$ be its (bi)-adjoint.

Then there are exactly $n$ isomorphism classes of Lagrangian algebras in $\mathcal{Z}(\mathcal{D})$, given by $\{A_{i}:=I_{i}(\mathbbm{1}_{\mathcal{D}_{ii}})\}^{n}_{i=1}$, with $A_{1}$ being the canonical Lagrangian algebra in $\mathcal{Z}(\mathcal{D})$. We can express $A_{i}$ as the pair $(\bigoplus_{Y\in \mathcal{D}_{1j}} Y\otimes \overline{Y}, \psi_{A_{i}})$ where

$$\psi_{A_{i},W}:= 
   \bigoplus_{Y,Z\in\text{Irr}(\mathcal{D}_{1i})}
    \sum_i \frac{\sqrt{d_Z}}{\sqrt{d_Y}}
    \tikzmath{
        \draw[thick]
        (0,0) node [below] {$\scriptstyle Y$}--(0,4)node [above] {$\scriptstyle Z$ }
        (0,2) to [in=270 ,out=110] (-1,4) node [above] {$\scriptstyle W$}
        (2,0)node [below] {$\scriptstyle \bar Y$ }--(2,4) node [above] {$\scriptstyle \bar Z$ }
        (2,2) to [out=290, in=90] (3,0) node [below] {$\scriptstyle W$}
        ;
        \mydotw{(0,2)}
        \mydotw{(2,2)}
        \node at (1,2){$\scriptstyle i$};
        \node at (0,2) [left] {$\alpha$};
        \node at (2,2) [right] {$\alpha^{\bullet}$};}\,.
        $$

Here $i$ is labelling the region, meaning we project onto the $i^{th}$ summand of $\mathbbm{1}$ in $\mathcal{D}$ (we assume a blank label indicates the region is labelled by $1$). $\{\alpha\}$ is a basis for $\text{Hom}_{\cD_{1i}}(Y, W\otimes Z)$ with $\beta^{\dagger}\circ \alpha=\delta_{\alpha,\beta} 1_{Y}$, and $\{\alpha^{\bullet}\}\subseteq \mathcal{D}_{i1}(\bar{Y}\otimes W, \bar{Z})$

$$\alpha^{\bullet}:=(\text{ev}_{Y} \otimes 1_{\overline{Z}}) \circ  \alpha^{\dagger} \circ (1_{\overline{Y}}\otimes 1_{W}\otimes \text{coev}_{Z})\,.$$

Each object $A_{i}$ is canonically endowed with the structure of a Q-system in $\cZ(\mathcal{D})$, with structure maps
\begin{align*}
    \tikzmath{
        \draw[very thick]
            (0,.5) node [below] {$\scriptstyle L_{i}$ } --(0,1) arc (180:0:1) --(2,.5) node [below] {$\scriptstyle L_{i}$ } 
            (1,2)--(1,4) node [above]
            {$\scriptstyle L_{i}$ }
            ;            
            \mydotw{(1,2)}
    }
    &=
    \bigoplus_{X\in\text{Irr}(\mathcal{D}_{1i})}
    \frac{1}{\sqrt{d_{X}}}
    \tikzmath{
        \draw[thick] 
            (0,0.5) node [below] {$\scriptstyle \bar X$ } -- (0,1) arc (180:0:1) --(2,.5)
            node [below] {$\scriptstyle X$ }
            (-1,.5) node [below] {$\scriptstyle  X$} -- (-1,1) arc (180:135:2)
            to [in=-90,out=45]
            (.5,3.5) --(.5,4) node [above] {$\scriptstyle X$}
              (3,.5) node [below] {$\scriptstyle  \bar X$}  -- (3,1) arc (0:45:2) to [in=-90,out=135] (1.5,3.5)--(1.5,4) node [above] {$\scriptstyle \bar  X$}
            ;\node at (1,3){$\scriptstyle i$};}
    \,,
    \\
     \tikzmath{
    \draw[very thick]
    (0,0) to (0,2) node [above] {$\scriptstyle L_{i}$};
    \mydotw{(0,0)};
    }
    &=
    \bigoplus_{ X\in \text{Irr}(\mathcal{D}_{1i})} \sqrt{d_{X}}\tikzmath{
    \draw[thick]
    (2,1)node [above] {$\scriptstyle \bar{X}$} --(2,0) 
    arc(360:180:1) 
    (0,0)--(0,1) node [above] {$\scriptstyle X$};
    \node at (1,0.5){$\scriptstyle i$};\,.
    }
\end{align*}
The comultiplication and counit are given by the reflected diagrams of the multiplication and unit maps respectively, with the same normalizing coefficients. Thus $A_{i}$ is a connected Q-system in $\mathcal{Z}(\mathcal{D})$.
From above we see
$$
    \mathcal{Z}(\mathcal{D})(A_{i}, A_{j})\cong \bigoplus_{X\in\text{Irr}(\mathcal{D}_{ij})} \mathcal{D}_{ij}(\mathbb{1}, X\otimes \bar{X})\,.
$$

Now for $Y\in \text{Irr}(\mathcal{D}_{ij})$, define
\begin{equation}\label{eYdefinition}
    \Phi_Y=
    \bigoplus_{X\in \text{Irr}(\mathcal{D}_{1j}), Z\in \text{Irr}(\mathcal{D}_{1i})}
    \sum_i \frac{\sqrt{d_{X}}\sqrt{d_{Z}}}{d_{Y}}
    \tikzmath{
        \draw[thick]
        (0,0) node [below] {$\scriptstyle  X$ }--(0,4)node [above] {$\scriptstyle Z$ }
        (0,1) to [in=260,out=80] 
        node [above] {$\scriptstyle Y$}
        (2,3)
        (2,0)node [below] {$\scriptstyle \bar X$ }--(2,4) node [above] {$\scriptstyle \bar Z$ }
        ;
        \mydotw{(0,1)}
        \mydotw{(2,3)}
         \node at (0,1) [right] {$\scriptstyle \alpha^{\bullet}$};
        \node at (2,3) [right] {$\scriptstyle \alpha$};
    }\,
\end{equation}
where $\alpha$ ranges over a basis for $\widetilde{\cD}(Y\otimes \bar{X},\bar{Z})$. It is easy to verify that $\Phi_{Y}\in \mathcal{Z}(\mathcal{D})(A_{j}, A_{i})$.  Comparing this with $\cite{BJ21}$, we note that our $\Phi_{Y}=\frac{1}{d^{2}_{Y}} e_{Y}$ in their notation. With this dictionary in mind, then following the calculations in \cite[Section 3.1]{BJ21} we have the following.

\begin{theorem}

\begin{enumerate}
    \item 
    $\{\Phi_{Y}\}_{Y\in \text{Irr}(\mathcal{D}_{ij})}$ forms a linear basis for $\mathcal{Z}(\mathcal{D})(L_{j},L_{i})$.
    \item
    For $Y,Z\in \text{Irr}(\mathcal{D}_{ij})$, $\Phi_{Y}\ast \Phi_{Z}=\delta_{Y,Z} \frac{d^{2}_{X}}{d^{2}_{Y}d^{2}_{Z}} \Phi_{Y}$.
    \item
    For $Y\in \text{Irr}(\mathcal{D}_{ij})$ and $Z\in \text{Irr}(\mathcal{D}_{jk})$ $\Phi_{Y}\circ \Phi_{Z}=\sum_{X\in \text{Irr}(\mathcal{D}_{ik})} \frac{d_{X}}{d_{Y}d_{Z}} N^{X}_{YZ} \Phi_{X}$. 
\end{enumerate}

\end{theorem}

\begin{prop}
Let $B\subseteq A_{1}$ and $B\subseteq A_{2}$ be topological bounday extensions.

\begin{enumerate}
    \item 
There is a canonical bijective assignment $\text{Fun}_{\mathcal{C}}(\mathcal{M}_{A_{1}}, \mathcal{M}_{A_{2}})\rightarrow \text{Ch}_{B}(A_{1}, A_{2})$, $X\mapsto \Phi_{X}$ between the extreme points $\text{Ch}_{B}(A_{1}, A_{2})$ and equivalence classes of simple $\mathcal{C}$-module functors in $\text{Fun}_{\mathcal{C}}(\mathcal{M}_{A_{1}}, \mathcal{M}_{A_{2}})$.
\item 
If $A_{1}$, $A_{2}$ and $A_{3}$ are Lagrangian algebras, let $X\in \text{Fun}_\cC(\mathcal{M}_{A_1},\mathcal{M}_{A_2}),\ Y\in \text{Fun}_\cC(\mathcal{M}_{A_2},\mathcal{M}_{A_3})$ and $\Phi_{X}\in \text{Ch}_{B}(A_{1},A_{2}), \Phi_{Y}\in \text{Ch}_{B}(A_{2},A_{3})$ the associated channels. Then $$\Phi_{X}\circ \Phi_{Y}=\sum_{Z\in \text{Irr}\left(\text{Fun}_{\mathcal{C}}(\mathcal{M}_{A_{1}},\mathcal{M}_{A_3})\right)} \frac{d_{Z}}{d_{X}d_{Y}} N^{Z}_{X,Y} \Phi_{Z}$$

where $N^{Z}_{X,Y}$ denotes the fusion rule for module functors.
\end{enumerate}
\end{prop}

\subsubsection{Basis decomposition}\label{sec:basis}

If we take a module fucntor $X:\mathcal{M}_{1}\rightarrow \mathcal{M}_{2}$, the cp map between the associated Lagrangian algebras $A_{1}$ and $A_{2}$ is denoted by $\Phi_{X}$ as above.

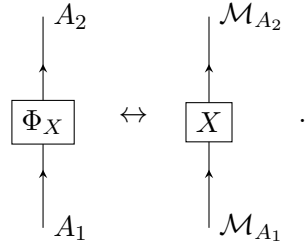
\[
\diagram{1.4}{
\draw[string] (0,-1)node[right]{$A_1$}--(0,0);
\draw[string] (0,0)--(0,1)node[right]{$A_2$};
\node[draw = black, fill = white] at (0,0) {$\Phi_X$};
}
\;\, \leftrightarrow \quad
\diagram{1.4}{
\draw[string] (0,-1)node[right]{$\cM_{A_1}$}--(0,0);
\draw[string] (0,0)--(0,1)node[right]{$\cM_{A_2}$};
\node[draw = black, fill = white] at (0,0) {$X$};
}\,.
\]
Given that $A$ is a commutative Q-system in $\DHR(B) \cong_{br} \cZ(\cD)$, for $X \in \cD$, we can write down the local operators explicitly (for example, see \cite{jones2025quantumcellularautomatacategorical}):
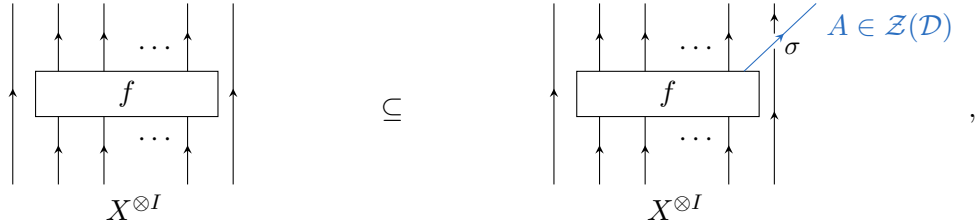
\[
\raisebox{-11 pt}{
\begin{tikzpicture}[baseline={($(current bounding box)$)}]
        \draw[string] (0,-1.2) --  (0,-.3);
        \draw[string] (0,.3) --  (0,1.2);
        \filldraw[fill=white] (-.3,-.3) rectangle node{$f$} (2.1,.3);
        \draw[string] (.6,-1.2) --  (.6,-.3);
        \draw[string] (.6,.3) --  (.6,1.2);
        \draw[string] (-.6,-1.2) --  (-.6,1.2);
        \draw[string] (2.3,.-1.2) --  (2.3,1.2);
        \node at (1.3,.6) {$\cdots$};
        \node at (1.3,-.6) {$\cdots$};
        \node at (1,-1.5){$X^{\otimes I}$};
        \node at (1,-2.3){$\text{symmetric local operators in }B$};
        \draw[string] (1.7,-1.2) --  (1.7,-.3);
        \draw[string] (1.7,.3) --  (1.7,1.2);
\end{tikzpicture}} \quad \subseteq \quad \quad \quad \quad
\raisebox{-11 pt}{
\begin{tikzpicture}[baseline={($(current bounding box)$)}]
        \draw[string] (0,-1.2) --  (0,-.3);
        \draw[string] (0,.3) --  (0,1.2);
        \filldraw[fill=white] (-.3,-.3) rectangle node{$f$} (2.1,.3);
        \draw[string] (.6,-1.2) --  (.6,-.3);
        \draw[string] (.6,.3) --  (.6,1.2);
        \draw[string] (-.6,-1.2) --  (-.6,1.2);
        \draw[string] (2.3,.-1.2) --  (2.3,.6)node[right]{\small$\sigma$};
        \draw[string] (2.3,.8) --  (2.3,1.2);
        \node at (1.3,.6) {$\cdots$};
        \node at (1.3,-.6) {$\cdots$};
        \node at (1,-1.5){$X^{\otimes I}$};
        \draw[string] (1.7,-1.2) --  (1.7,-.3);
        \draw[string] (1.7,.3) --  (1.7,1.2);
        \draw[string, YaleMidBlue] (1.9,.3) -- (2.85,1.2)node[below right]{$A \in \cZ(\cD)$};
        \node at (1,-2.3){$\text{local operators in }A$};
\end{tikzpicture}}\,,
\]
where $\sigma_{A,X}: A \ot X \xrightarrow[]{\simeq} X \ot A $ is half-braiding in $\cZ(\cD)$.
Combining the bimodule channels with the local operators, we have
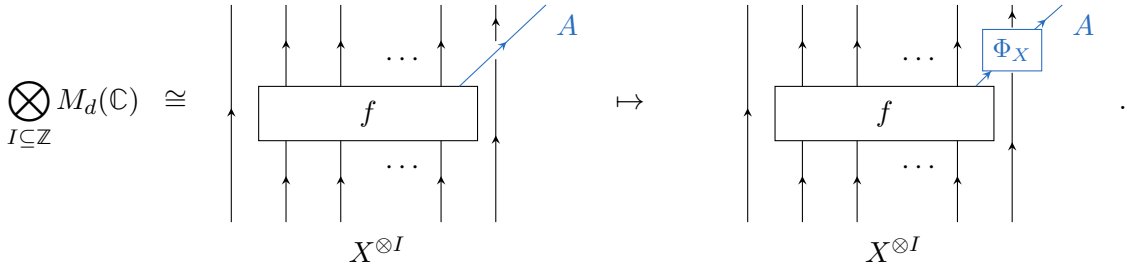
\[
\bigotimes_{I \subseteq \Z} M_d(\C)\;\; \cong \;\;
\raisebox{-11 pt}{
\diagram{1.2}{
\draw[string] (0,-1.2) --  (0,-.3);
        \draw[string] (0,.3) --  (0,1.2);
        \filldraw[fill=white] (-.3,-.3) rectangle node{$f$} (2.1,.3);
        \draw[string] (.6,-1.2) --  (.6,-.3);
        \draw[string] (.6,.3) --  (.6,1.2);
        \draw[string] (-.6,-1.2) --  (-.6,1.2);
        \draw[string] (2.3,.-1.2) --  (2.3,.6);
        \draw[string] (2.3,.8) --  (2.3,1.2);
        \node at (1.3,.6) {$\cdots$};
        \node at (1.3,-.6) {$\cdots$};
        \node at (1,-1.5){$X^{\otimes I}$};
        \draw[string] (1.7,-1.2) --  (1.7,-.3);
        \draw[string] (1.7,.3) --  (1.7,1.2);
        \draw[string, YaleMidBlue] (1.9,.3) -- (2.85,1.2)node[below right]{$A $};
}
}
 \mapsto\quad \quad \;\;
\raisebox{-11 pt}{
\diagram{1.2}{ 
\draw[string] (0,-1.2) --  (0,-.3);
        \draw[string] (0,.3) --  (0,1.2);
        \filldraw[fill=white] (-.3,-.3) rectangle node{$f$} (2.1,.3);
        \draw[string] (.6,-1.2) --  (.6,-.3);
        \draw[string] (.6,.3) --  (.6,1.2);
        \draw[string] (-.6,-1.2) --  (-.6,1.2);
        \draw[string] (2.3,.-1.2) --  (2.3,.45);
        \draw[string] (2.3,.95) --  (2.3,1.2);
        \node at (1.3,.6) {$\cdots$};
        \node at (1.3,-.6) {$\cdots$};
        \node at (1,-1.5){$X^{\otimes I}$};
        \draw[string] (1.7,-1.2) --  (1.7,-.3);
        \draw[string] (1.7,.3) --  (1.7,1.2);
        \draw[string, YaleMidBlue] (2.4,.8) -- (2.85,1.2)node[below right]{$A $};
        \draw[string, YaleMidBlue] (1.9,.3) -- (2.15,.55);
        \node[draw = YaleMidBlue, fill = white, text = YaleMidBlue] at (2.3,.7) {\small$\Phi_X$};
}}\,.
\]
See \eqref{eYdefinition} for the explicit formula of the basis $\Phi_X$.

\subsection{Parameterizing duality operators }\label{sec:parametrization}

Let $\alpha\in \text{QCA}(B)$. Then by \cite{jones2024dhrbimodulesquasilocalalgebras}, $\alpha$ acts by braided autoequivalences on $\text{DHR}$ by twisting. Given $X\in \text{DHR}(B)$, define

$X^{\alpha}:=X$ as a vector space with

$$a\triangleright x\triangleleft b:=\alpha^{-1}(a)x\alpha^{-1}(b)$$

$$\langle x\ |\ y\rangle^{\alpha}:=\alpha(\langle x\ |\ y\rangle).$$

\noindent At the level of morphisms, if $f:X\rightarrow Y$ is an intertwiner, then $\widetilde{\alpha}(f):=f$ as a set function, which is easily seen to be an intertwiner between the twisted bimodules.

By \cite{jones2024dhrbimodulesquasilocalalgebras}, this is again a DHR bimodule, and defines a braided autoequivalence of $\text{DHR}(B)$, which we denote $\widetilde{\alpha}$.

Now, let $\text{Ch}^{\alpha}_{B}(A_{1},A_{2})$ denote the cp maps $\Psi: A\rightarrow A$ with $\Psi|_{B}=\alpha$. We have the following result (c.f. \cite{jones2025quantumcellularautomatacategorical}).

\begin{theorem}\label{parameterizingduality}
    Let $B$ be a physical boundary algebra, and let $\alpha\in \text{QCA}(B)$. Suppose $B\subseteq A_{1}, B\subseteq A_{2}, B\subseteq A_{3}$ are topological boundaries, so that $A_{1}, A_{2}, A_{3}$ are Lagrangian algebras in $\text{DHR}(B)$. 
    
    \begin{enumerate}
    \item 
    Then there is a canonical (convex) isomorphism  $\lambda:\text{Ch}^{\alpha}_{B}(A_{1},A_{2})\cong \text{Ch}_{B}(\widetilde{\alpha}(A_{1}), A_{2})$. 
    \item 
    If $\Psi'\in \text{Ch}^{\alpha}_{B}(A_{1},A_{2})$ and $\Psi\in \text{Ch}^{\beta}_{B}(A_{2},A_{3})$, then $$\Psi\circ \Psi'=\lambda^{-1}((\lambda(\Psi)\circ \widetilde{\beta}(\lambda(\Psi')))\in \text{Ch}^{\beta\circ \alpha}_{B}(A_{1},A_{3}).$$
    
    \end{enumerate}
\end{theorem}

\begin{proof}
Let $\Psi\in \text{Ch}^{\alpha}_{B}(A_{1},A_{2})$. Then considering $A_{1}$ as a DHR bimodule of $B$, we claim the map of vector spaces $\Psi: A_{1}\rightarrow A_{2}$ actually defines a $B$-$B$ bimodule map 
$\widetilde{\alpha}(A_{1})\rightarrow A_{2}$, which we will denote $\lambda(\Psi)$. We compute

$$\Psi(b_{1}\triangleright a\triangleleft b_{2})=\Psi(\alpha^{-1}(b_1)a\alpha^{-1}(b_2))=b_{1}\Psi(a)b_{2}.$$

Thus $\Psi\in \text{Ch}_{B}(\widetilde{\alpha}(A_{1}), A_{2})$. Conversely the same computation shows that for any $\Psi\in \text{Ch}_{B}(\widetilde{\alpha}(A_{1}), A_{2})$, the underlying map of vector spaces $\Psi: A_{1}=\widetilde{\alpha}(A_{1}) \rightarrow A_{2}$ \textit{is} in $\text{Ch}_{B}(\widetilde{\alpha}(A_{1}), A_{2})$. The statement concerning composition follows trivially.

\end{proof}

The above theorem can be represented graphically by asserting that we always have the decomposition

\[\text{Ch}^\alpha_B(A_1,A_2):\quad
\diagram{1.4}{
\draw (0,-.5)--(2.6,-.5);
\draw (0,1)--(2.6,1);
\fill[YaleLightGrey] (0,-.5) rectangle  (2.6,1);
\draw[YaleDarkGrey] (1.3,-.5)--(1.3,1);
\node at (.65,-.8) {$B$};
\node at (1.95,-.8) {$B$};
\node at (.65,1.3) {$A_1$};
\node at (1.95,1.3) {$A_2$};
\node at (1.3,-.8) {\small$\alpha$};
\fill(1.3,-.5) circle [radius = .05];
\fill(1.3,1) circle [radius = .05];
\node at (.3,-1.4){$\longrightarrow$};
\node at (.3,-1.6){\small$\text{time}$};
}\quad =\quad
\diagram{1.4}{
\draw (0,-.5)--(3.6,-.5);
\draw (0,1)--(3.6,1);
\fill[YaleLightGrey] (0,-.5) rectangle  (3.6,1);
\draw[YaleDarkGrey] (1.2,-.5)--(1.2,1);
\draw[trivial] (2.4,-.5)--(2.4,1);
\node at (.6,-.8) {$B$};
\node at (1.8,-.8) {$B$};
\node at (3,-.8) {$B$};
\node at (.6,1.3) {$A_1$};
\node at (1.8,1.3) {$\widetilde{\alpha}(A_1)$};
\node at (3,1.3) {$A_2$};
\node at (1.2,-.8) {\small$\alpha$};
\node at (2.4,-.8) {\small$\id_B$};
\fill(1.2,-.5) circle [radius = .05];
\fill(1.2,1) circle [radius = .05];
\fill(2.4,-.5) circle [radius = .05];
\fill(2.4,1) circle [radius = .05];
\node at (.3,-1.4){$\longrightarrow$};
\node at (.3,-1.6){\small$\text{time}$};
}\,.
\]

\noindent Theorem 1.1 from the introduction now follows.

\bigskip

\textit{Proof of Theorem 1.1}. For any $\alpha\in \text{QCA}(B)$, we have an autoequivalence $\widetilde{\alpha}\in \text{Aut}_{br}(\mathcal{Z}(\mathcal{C}))$. By \cite{etingof2009fusioncategorieshomotopytheory}, this defines an invertible $\mathcal{C}$-$\mathcal{C}$ bimodule category, which as a right $\mathcal{C}$ module category corresponds to $\widetilde{\alpha}^{-1}(I(\mathbb{1}))$. The simple objects of this bimodule category are in bijective correspondence with the minimal convolution idempotents of the algebra $H(I(\mathbb{1}), \alpha^{-1}(I(\mathbb{1}))\cong H(\alpha(I(\mathbb{1})), I(\mathbb{1}))$. Thus the result follows.


\subsection{Universal tensor categories and emanant symmetries}\label{sec:universal_cat}

We turn our attention to characterizing emanant symmetries of duality operators. We assume we have a physical boundary subalgebra $B\subseteq A$ for a fixed boundary, with symmetry category $\mathcal{C}$. We want to understand UV-categories of duality operators on $A$, which will flow to internal symmetries under any $\mathcal{C}$-constained RG. In general these categories will not be fusion, in particular if the restriction of $\Psi|_{B}\in \text{QCA}(B)$ does not have finite order. 

We have already characterized the duality operators that sit over a fixed QCA, $\alpha\in \text{QCA}(B)$, and we have a formula to compute how these compose, given also by the previous composition. If $\Psi'\in \text{Ch}^{\alpha}_{B}(A, A)$ and $\Psi\in \text{Ch}^{\beta}_{B}(A,A)$, then $\Psi'\circ \Psi\in \text{Ch}^{\alpha\circ \beta}_{B}(A,A)$, so it will decompose as a convex combination, and thus this looks like a kind of (normalized) fusion rule. 

We will see there is a universal unitary tensor category here, though it is not necessarily fusion.

\begin{theorem}\label{thm:UV-category}
Let $B\subseteq A$ be a physical boundary subalgebra, and let $\{\alpha_{j}\}^{n}_{j=1}\subseteq \text{QCA}(B)$. Let $\pi: \mathbb{F}_{n}\rightarrow \text{QCA}(B)$ be the homomorphsim of groups sending the jth generator to $\alpha_{j}$.

\begin{enumerate}
\item
There is a canonical $\mathbb{F}_{n}$-graded extension of $\mathcal{C}$, denoted $\mathcal{C}\#\mathbb{F}_{n}$. 
\item 
For each element $g\in \mathbb{F}_{n}$, there is a bijection $X\leftrightarrow\Phi_{X}$ between the simple objects $X\in \left(\mathcal{C}\#\mathbb{F}_{n}\right)_{g} $ of the $g$-graded component and the extreme points of the simplex $\text{Ch}^{\pi(g)}_{B}(A,A)$.
\item 
If $X\in \left(\mathcal{C}\#\mathbb{F}_{n}\right)_{g}$ and $Y\in \left(\mathcal{C}\#\mathbb{F}_{n}\right)_{h}$, then $\Phi_{X}\circ \Phi_{Y}:=\sum_{Z\in \left(\mathcal{C}\#\mathbb{F}_{n}\right)_{gh}} \frac{d_{Z}}{d_{X}d_{Y}}N^{Z}_{XY} \Phi_{Z}$.
\end{enumerate} 
\end{theorem}

\begin{proof}
We note that since $\text{QCA}(B)\le \text{Aut}(B)$, the action on bimodules is just by conjugation, and the action of QCA on $\text{DHR}$ bimodules is just the restriction of this general action, we get a 2-group homomorphism $\underline{\text{QCA}(B)}\rightarrow \underline{\text{Aut}}_{br}(\mathcal{Z}(\mathcal{C}))$. Composing $\pi$ yields a 2-group homomorphism $\widetilde{\pi}:\underline{\mathbb{F}_{n}}\rightarrow \underline{\text{Aut}}_{br}(\mathcal{Z}(\mathcal{C}))$. By the standard extension theory of fusion categories \cite{etingof2009fusioncategorieshomotopytheory}, the existence of an $\mathbb{F}_{n}$-graded extension is obstructed by a class $o_{4}(\widetilde{\pi})\in \text{H}^{4}(\mathbb{F}_{n}, \text{U}(1))$, and if this vanishes, the solutions to the pentagon equation for the associator form a torsor over $H^{3}(\mathbb{F}_{n}, \text{U}(1))$. Both of these cohomology groups are trivial for free groups, and thus there exists a unique $\mathbb{F}_{n}$-graded extension. We will call this category $\mathcal{C}\# \mathbb{F}_{n}$

Now, given an $\alpha\in \text{QCA}(A,A)$, we have an isomorphism $\text{Ch}^{\alpha}_{B}(A,A)\cong \text{Ch}_{B}(\alpha(A),A)$, and by \cite[Section 3.2]{BJ21} the extreme points correspond to the simple objects in $\alpha$ graded component (or more precisely, the $g$ graded component for any $g\in \pi^{-1}(\alpha)$. Then the claim about composition follows from part 2 of Theorem \ref{parameterizingduality}.

\end{proof}

This gives us Theorem \ref{Theorem:UniversalCat}

\bigskip

\textit{Proof of Corollary \ref{Cor:WeaklyIntegral}}.

Any category $\cX$ obtain under RG from an emanant symmetry will necessarily be a quotient of $\mathcal{C}\# \mathbb{F}_n$, and thus will be a G-graded quotient of a quotient of $\mathcal{C}$, where $G$ is some quotient of $\mathbb{F}_{n}$. But by \cite{evans2025operatoralgebraicapproachfusion}, $\mathcal{C}$ must be integral, hence any quotient of $\mathcal{C}$ is integral \cite[Lemma 3.5.6]{etingof2015tensor}. Thus $\cX$ is a $G$-graded extension of an integral category, which is neccessarily weakly integral. Indeed, if $X\in \cX$ is a simple object, then $d(X)^{2}=d(X\otimes X^{*})$, but $X\otimes X^{*}$ is in the quotient of $\mathcal{C}$, which is integral.

\subsubsection{Example: Tambara-Yamagami type duality operators}
The simpliest example is $\cC:=\Hilb_{A}$, for $A$ being a finite abelian group, which has been studied widely in ZX calculus \cite{Tantivasadakarn_2023,Tantivasadakarn_2024,gorantla2024tensornetworksnoninvertiblesymmetries,lu2025strangecorrelatorstringorder} and tensor network \cite{Lootens_2023,Lootens_2024}. We revisit this class of examples to point out the precise UV-category being a $\Z$-graded extension of $A$ and extra intertwiners in computing the $F$-symbols.

As discussed in \cite{jones2025quantumcellularautomatacategorical}, given a symmetric non-degenerate bicharacter $\chi:A \times A \rightarrow \text{U}(1)$, it induces an isomorphsim 
\[
\widetilde{\chi}:A \rightarrow \hat{A}, \quad \widetilde{\chi}(a)(x) = \chi(a,x)\,.
\]
The objects in the center $\cZ(\cC) \cong \Hilb_{A \times \hat{A}}$ are labeled by pairs $(a,\phi)$ for $a \in A, \phi \in \hat{A}$ giving the half-braiding $\sigma_{x}:=\phi(x)\,\id_{ax}$. The braided autoequivalence $\widetilde{\alpha}_\chi \in \Aut_{br}(\Hilb_{A \times \hat{A}})$ is given by
\[
\widetilde{\alpha}_\chi(a,\phi):=\left(\widetilde{\chi}^{-1}(\phi),\widetilde{\chi}(a)\right)\,.
\]
This gives the $\Z_2$-graded extension of $A$, i.e. the well-known Tambara-Yamagami category \cite{TAMBARA1998692} with simple objects being $A \cup \{m\}$.

$X := \bigoplus_{a \in A}a$ is a strongly tensor generator of $\cC$, then $\DHR(\alpha_\chi) \cong \widetilde{\alpha}_\chi$ gives a QCA acting on the symmetric subalgebra by swapping the two canonical Lagrangian algebras $L_1=\C[ A] \times \one$ and $L_2=\one \times \C[\hat{A}]$. 

When $A =\Z_2$, such a symmetric non-degenerate bicharacter is unique. We diagrammatically represent the duality operator as \cite{lu2025strangecorrelatorstringorder,inamura2026remarksnoninvertiblesymmetriestensor}
\begin{align*}
    & m = \diagram{1}{
    \draw[string](-1,0)--(0,0);
    \draw[string](0,0)--(1,0);
    \draw[string](0,0)--(0,1);
    \filldraw[fill=white](0,0) circle [radius=0.1];
    } \;, \quad \quad m^\dagger =\diagram{1}{
    \draw[string](-1,0)--(0,0);
    \draw[string](0,0)--(1,0);
    \draw[string](0,-1)--(0,0);
    \filldraw[fill=black](0,0) circle [radius=0.1];
    }\;, \quad \quad \sH = \diagram{1}{
    \draw[string](-1,0)--(0,0);
    \draw[string](0,0)--(1,0);
    \node[draw = black, fill = white] at (0,0){} ;
    }= \frac{1}{\sqrt{2}}\; \chi(a,b) \\[1.5 ex]  
    & \sD_{-} = \diagram{1}{
    \draw[string](0,0)--(.7,0);
    \draw[string](-.7,0)--(0,0);
    \draw[string](.7,0)--(1.4,0);
    \draw[string](0,-.7)--(0,0);
    \draw[string](1.4,0)--(2.1,0);
    \draw[string](2.1,0)--(2.8,0);
    \draw[string](1.4,0)--(1.4,.7);
    \filldraw[fill=black](0,0) circle [radius=0.1];
    \filldraw[fill=white](1.4,0) circle [radius=0.1];
    \node[draw = black, fill = white] at (.7,0){};
    \node[draw = black, fill = white] at (2.1,0){};
    }\;, \quad \quad \sD_{+} = \diagram{1}{
    \draw[string](0,0)--(.7,0);
    \draw[string](-.7,0)--(0,0);
    \draw[string](.7,0)--(1.4,0);
    \draw[string](1.4,0)--(2.1,0);
    \draw[string](2.1,0)--(2.8,0);
    \draw[string](0,0)--(0,.7);
    \draw[string](1.4,-.7)--(1.4,0);
    \filldraw[fill=white](0,0) circle [radius=0.1];
    \filldraw[fill=black](1.4,0) circle [radius=0.1];
    \node[draw = black, fill = white] at (.7,0){};
    \node[draw = black, fill = white] at (2.1,0){};
    }\,.
\end{align*}

$\alpha_\chi $ has infinite order in $\text{QCA}(B)$. By Theorem \ref{thm:UV-category}, the $\Z$-graded extension of $\Z_2$ exsits and is defined as
\begin{alignat*}{2}
	&\mathrm{Irr}(\cC_0) = \{1,\eta\}, &&\\
	&\mathrm{Irr}(\cC_1) = \{\sD_{+}\},&& \quad \mathrm{Irr}(\cC_{-1}) = \{\sD_{-}\},\\
	&\mathrm{Irr}(\cC_2) = \{T^+, \eta T^+\},&&\quad \mathrm{Irr}(\cC_{-2}) = \{T^-,\eta T^-\},\\
	&\mathrm{Irr}(\cC_3) = \{\sD_{+}T^+\},&& \quad \mathrm{Irr}(\cC_{-3}) = \{\sD_{-}T^-\},\\
	&\mathrm{Irr}(\cC_4) = \{T^+ T^+, \eta T^+ T^+\},&&\quad \mathrm{Irr}(\cC_{-4}) = \{T^- T^-,\eta T^- T^-\}, \quad \cdots\\		
\end{alignat*}
where the objects are defined as \footnote{We use the convention in \cite{Seiberg_2024}.}
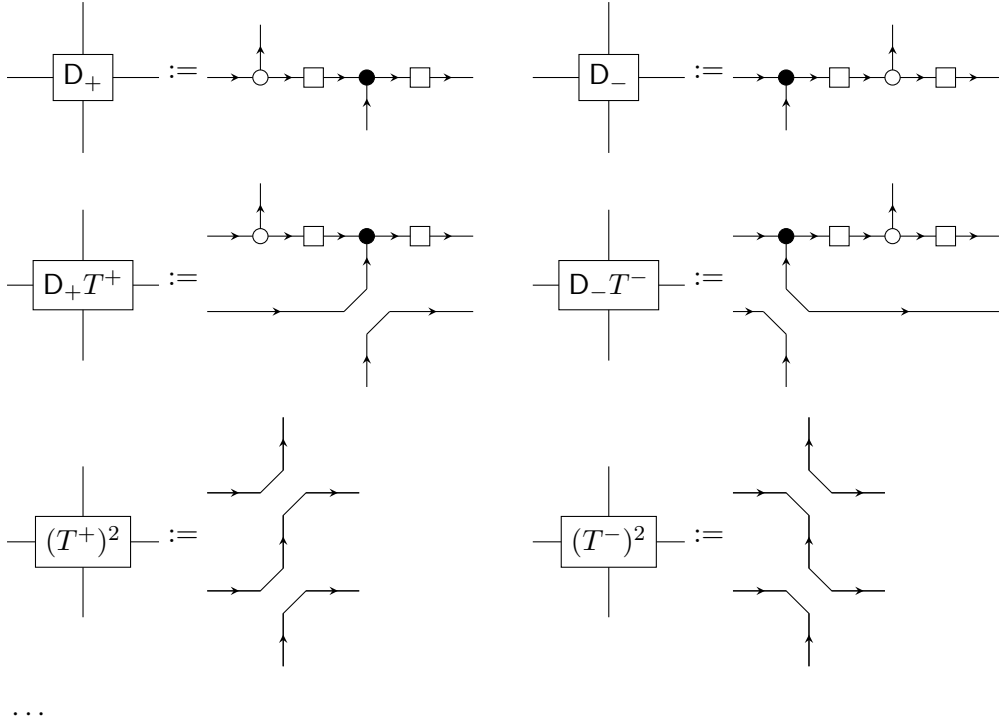
\begin{alignat*}{2}
&\diagram{1}{
\draw (0,-1) -- (0,1);
\draw (1,0) -- (-1,0);
\node[draw = black, fill = white] at (0,0) {$T^+$};
} := \diagram{1}{
\draw[string](-1,0)--(-.3,0);
\draw[string](.3,0)--(1,0);
\draw[string](0,-1)--(0,-.3);
\draw[string](0,.3)--(0,1);
\draw(-1,0)--(-.3,0)--(0,.3)--(0,1);
\draw(0,-1)--(0,-.3)--(.3,0)--(1,0);
} && \qquad \diagram{1}{
\draw (0,-1) -- (0,1);
\draw (1,0) -- (-1,0);
\node[draw = black, fill = white] at (0,0) {$T^-$};
} := \diagram{1}{
\draw[string](-1,0)--(-.3,0);
\draw[string](.3,0)--(1,0);
\draw[string](0,-1)--(0,-.3);
\draw[string](0,.3)--(0,1);
\draw(1,0)--(.3,0)--(0,.3)--(0,1);
\draw(0,-1)--(0,-.3)--(-.3,0)--(-1,0);
} \\[1.5 ex]
&\diagram{1}{
\draw (0,-1) -- (0,1);
\draw (1,0) -- (-1,0);
\node[draw = black, fill = white] at (0,0) {$ \eta T^+$};
} := \diagram{1}{
\draw[string](-1,0)--(-.3,0);
\draw[string](.3,0)--(1,0);
\draw[string](0,-1)--(0,-.3);
\draw[string](0,.3)--(0,1.5);
\draw(-1,0)--(-.3,0)--(0,.3)--(0,1);
\draw(0,-1)--(0,-.3)--(.3,0)--(1,0);
\node[draw = black, fill = white] at (0,.8) {$X$};
} && \qquad \diagram{1}{
\draw (0,-1) -- (0,1);
\draw (1,0) -- (-1,0);
\node[draw = black, fill = white] at (0,0) {$\eta T^-$};
} := \diagram{1}{
\draw[string](-1,0)--(-.3,0);
\draw[string](.3,0)--(1,0);
\draw[string](0,-1)--(0,-.3);
\draw[string](0,.3)--(0,1.5);
\draw(1,0)--(.3,0)--(0,.3)--(0,1);
\draw(0,-1)--(0,-.3)--(-.3,0)--(-1,0);
\node[draw = black, fill = white] at (0,.8) {$X$};
} \\[1.5 ex]
& \diagram{1}{
\draw (0,-1) -- (0,1);
\draw (1,0) -- (-1,0);
\node[draw = black, fill = white] at (0,0) {$\sD_{+}$};
}
:= \diagram{1}{
    \draw[string](0,0)--(.7,0);
    \draw[string](-.7,0)--(0,0);
    \draw[string](.7,0)--(1.4,0);
    \draw[string](1.4,0)--(2.1,0);
    \draw[string](2.1,0)--(2.8,0);
    \draw[string](0,0)--(0,.7);
    \draw[string](1.4,-.7)--(1.4,0);
    \filldraw[fill=white](0,0) circle [radius=0.1];
    \filldraw[fill=black](1.4,0) circle [radius=0.1];
    \node[draw = black, fill = white] at (.7,0){};
    \node[draw = black, fill = white] at (2.1,0){};
    }
&& \qquad \diagram{1}{
\draw (0,-1) -- (0,1);
\draw (1,0) -- (-1,0);
\node[draw = black, fill = white] at (0,0) {$\sD_{-}$};
} :=
\diagram{1}{
    \draw[string](0,0)--(.7,0);
    \draw[string](-.7,0)--(0,0);
    \draw[string](.7,0)--(1.4,0);
    \draw[string](0,-.7)--(0,0);
    \draw[string](1.4,0)--(2.1,0);
    \draw[string](2.1,0)--(2.8,0);
    \draw[string](1.4,0)--(1.4,.7);
    \filldraw[fill=black](0,0) circle [radius=0.1];
    \filldraw[fill=white](1.4,0) circle [radius=0.1];
    \node[draw = black, fill = white] at (.7,0){};
    \node[draw = black, fill = white] at (2.1,0){};
    }\\[1.5 ex]
& \diagram{1}{
\draw (0,-1) -- (0,1);
\draw (1,0) -- (-1,0);
\node[draw = black, fill = white] at (0,0) {$\sD_{+}T^+$};
}
:= \diagram{1}{
    \draw[string](0,0)--(.7,0);
    \draw[string](-.7,0)--(0,0);
    \draw[string](.7,0)--(1.4,0);
    \draw[string](1.4,0)--(2.1,0);
    \draw[string](2.1,0)--(2.8,0);
    \draw[string](0,0)--(0,.7);
    \draw[string](1.4,-.7)--(1.4,0);
    \draw(1.1,-1)--(1.4,-.7);
    \draw[string](-.7,-1)--(1.1,-1);
    \draw[string](1.4,-2)--(1.4,-1.3);
    \draw(1.4,-1.3)--(1.7,-1);
    \draw[string](1.7,-1)--(2.8,-1);
    \filldraw[fill=white](0,0) circle [radius=0.1];
    \filldraw[fill=black](1.4,0) circle [radius=0.1];
    \node[draw = black, fill = white] at (.7,0){};
    \node[draw = black, fill = white] at (2.1,0){};
    }
&& \qquad \diagram{1}{
\draw (0,-1) -- (0,1);
\draw (1,0) -- (-1,0);
\node[draw = black, fill = white] at (0,0) {$\sD_{-}T^-$};
} :=
\diagram{1}{
    \draw[string](0,0)--(.7,0);
    \draw[string](-.7,0)--(0,0);
    \draw[string](.7,0)--(1.4,0);
    \draw[string](0,-.7)--(0,0);
    \draw[string](1.4,0)--(2.1,0);
    \draw[string](2.1,0)--(2.8,0);
    \draw[string](1.4,0)--(1.4,.7);
    \draw(0,-.7)--(.3,-1);
    \draw[string](.3,-1)--(2.8,-1);
    \draw[string](0,-2)--(0,-1.3);
    \draw(-.3,-1)--(0,-1.3);
    \draw[string](-.7,-1)--(-.3,-1);
    \filldraw[fill=black](0,0) circle [radius=0.1];
    \filldraw[fill=white](1.4,0) circle [radius=0.1];
    \node[draw = black, fill = white] at (.7,0){};
    \node[draw = black, fill = white] at (2.1,0){};
    }\\[1.5 ex]
&\diagram{1}{
\draw (0,-1) -- (0,1);
\draw (1,0) -- (-1,0);
\node[draw = black, fill = white] at (0,0) {$(T^+)^2$};
} := \diagram{1}{
\draw(-1,-.5)--(-.3,-.5)--(0,-.2)--(0,.5)--(.3,.8)--(1,.8);
\draw(0,-1.5)--(0,-.8)--(.3,-.5)--(1,-.5);
\draw(-1,.8)--(-.3,.8)--(0,1.1)--(0,1.8);
\draw[string](0,-1.5)--(0,-.8);
\draw[string](0,-.2)--(0,.5);
\draw[string](0,1.1)--(0,1.8);
\draw[string](-1,-.5)--(-.3,-.5);
\draw[string](.3,.8)--(1,.8);
\draw[string](.3,-.5)--(1,-.5);
\draw[string](-1,.8)--(-.3,.8);
} 
&& \qquad \diagram{1}{
\draw (0,-1) -- (0,1);
\draw (1,0) -- (-1,0);
\node[draw = black, fill = white] at (0,0) {$(T^-)^2$};
} := \diagram{1}{
\draw(1,-.5)--(.3,-.5)--(0,-.2)--(0,.5)--(-.3,.8)--(-1,.8);
\draw(0,-1.5)--(0,-.8)--(-.3,-.5)--(-1,-.5);
\draw(1,.8)--(.3,.8)--(0,1.1)--(0,1.8);
\draw[string](0,-1.5)--(0,-.8);
\draw[string](0,-.2)--(0,.5);
\draw[string](0,1.1)--(0,1.8);
\draw[string](-1,-.5)--(-.3,-.5);
\draw[string](.3,.8)--(1,.8);
\draw[string](.3,-.5)--(1,-.5);
\draw[string](-1,.8)--(-.3,.8);
} \\[1.5 ex]
&\cdots &&
\end{alignat*}

We used the MPO's defined above to compute $F$-symbols and find that any extra $F$-symbols involving translation are trivial.  We started with determining the intertwiners $\iota \rho_{a} = \rho_{b \ot c} \,\iota$ for each $b \ot c = a$, where $a,b,c$ are simple objects in $\cC^\times_{\Z}$ \footnote{An intertwiner is a rank-3 tensor to reduce the composed local tensors.}. Note that as $\eta T^\bullet$ is defined as an object, we take the intertwiner $\eta T^\bullet \rightarrow T^\bullet \ot \eta$ into account:
\begin{equation}\label{diag:T_eta}
	\begin{split}
		&T^+ \ot \eta =\; \diagram{1}{
\draw(1.2,0)--(.3,0)--(0,-.3)--(0,-1.7);
\draw(0,1)--(0,.3)--(-.3,0)--(-1.2,0);
\draw[YaleMidBlue](1.2,0)--(1.5,0);
\draw[YaleMidBlue](-1.2,0)--(-1.5,0);
\node[draw = black, fill = white] at (0,-1) {$X$};
\node[draw = YaleMidBlue, fill = white, text = YaleMidBlue] at (0.9,0) {$X$};
\node[draw = YaleMidBlue, fill = white, text = YaleMidBlue] at (-0.9,0) {$X$};
} = \diagram{1}{
\draw(1.2,0)--(.3,0)--(0,-.3)--(0,-1);
\draw(0,1.7)--(0,.3)--(-.3,0)--(-1.2,0);
\node[draw = black, fill = white] at (0,1) {$X$};} = \eta T^+  \qquad \text{Intertwiner} \quad \diagram{1}{
\draw(-1,0)--(1,0);
\node[draw = black, fill = white] at (0,0) {$X$};
}\,;\\[1.5 ex]
&T^- \ot \eta =\; \diagram{1}{
\draw(1.2,0)--(.3,0)--(0,.3)--(0,1);
\draw(0,-1.7)--(0,-.3)--(-.3,0)--(-1.2,0);
\draw[YaleMidBlue](1.2,0)--(1.5,0);
\draw[YaleMidBlue](-1.2,0)--(-1.5,0);
\node[draw = black, fill = white] at (0,-1) {$X$};
\node[draw = YaleMidBlue, fill = white, text = YaleMidBlue] at (0.9,0) {$X$};
\node[draw = YaleMidBlue, fill = white, text = YaleMidBlue] at (-0.9,0) {$X$};
} = \diagram{1}{
\draw(1.2,0)--(.3,0)--(0,.3)--(0,1.7);
\draw(0,-1)--(0,-.3)--(-.3,0)--(-1.3,0);
\node[draw = black, fill = white] at (0,1) {$X$};} = \eta T^- \qquad \text{Intertwiner} \quad \diagram{1}{
\draw(-1,0)--(1,0);
\node[draw = black, fill = white] at (0,0) {$X$};
}\,.
	\end{split}
\end{equation}

To find a self-dual model, we can start with the trivial symmetric algebra $\Fun(e)$, the corresponding symmetric local operator is given by $(\iota\iota^\dagger)_i$ for $\iota: \Fun(e) \hookrightarrow \Fun(A)$. Then we act $\alpha_\chi \in \text{QCA}(B)$ (duality operator) on $\iota\iota^\dagger$, and have
\[
H=-\sum_{i\in \Z} \alpha_{\chi}(\iota\iota^\dagger)_i-\sum_{i \in \Z}(\iota\iota^\dagger)_i\,.
\]

When $A = \Z_n$, the Hamiltonian is
        \[
        H = -\sum_{i}\frac{1}{2}(1+Z^{(n)}_{i}Z^{(n)}_{i+1}) - \sum_{i}\frac{1}{2}(1+X^{(n)}_i)\,.
        \]

When $A = \Z_2 \times \Z_2$, $Q \in\Rep(\Z_2\times \Z_2)$, one self-dual Hamiltonian is 
        \[
        H = -\sum_{i \in 2\Z} \left( Z_{i-1} X_{i} Z_{i+1} + Z_{i} X_{i+1} Z_{i+2} \right) \,.
        \]
        The duality channels assembly into $\Z_2$-graded extension of $\Z_2 \times \Z_2$, i.e. $\TY_{\Z_2 \times \Z_2}^{\chi,\epsilon}$ isomorphic to $\Rep^\dagger(D_8)$, $\Rep^\dagger(Q_8)$ and $\Rep^\dagger(\sH_8)$ for different $(\chi,\epsilon)$'s. Note that $\TY_{\Z_2 \times \Z_2}^{\chi_{diag},\epsilon=-1}$ is isomorphic to an anomalous integral fusion category. The UV categorical structure of duality channels is determined by the order of QCA acting on $\Z_2 \times \Z_2$-symmetric subalgebra in the group QCA$(B)$. For a single duality, the general situation gives a $\Z$-graded extension of $\Z_2 \times \Z_2$, while a finite $\Z_n$-graded quotient can happen on the lattice if the QCA has finite order.

\bibliography{refs}
\bibliographystyle{alpha}
\end{document}